\documentclass{amsart}
\usepackage{amssymb,amsmath}
\usepackage{hyperref}
\usepackage{appendix}
\usepackage{paralist}
\usepackage{graphics}
\usepackage{epsfig}
\usepackage{graphicx}
\usepackage{epstopdf}
\usepackage[]{algorithm2e}
\usepackage{graphicx,amsxtra,amssymb, tikz}
\usepackage{amsmath}
\usepackage{epsfig}
\usepackage{hyperref}
\usepackage{appendix}
\usepackage{setspace}
\usepackage{MnSymbol,wasysym}
\newtheorem{thm}{Theorem}[section]
\newtheorem{remm}{Remark}[section]
\newtheorem{lem}{Lemma}[section]

\newtheorem{remark}[remm]{Remark}

\newtheorem{proposition}[thm]{Proposition}

\numberwithin{equation}{section}
\newcommand{\vq}{\vartheta}

\newcommand{\QQ}{Q}
\newcommand{\cset}{{\mathbb C}}
\newcommand{\rset}{{\mathbb R}}
\newcommand{\nset}{{\mathbb N}}

\newcommand{\ssy}{\scriptscriptstyle}

\newcommand{\uu}{u}

\newcommand{\UUU}{W}
\newcommand{\iii}{{\rm i}}
\newcommand{\mff}{{\mathfrak m}}
\newcommand{\gff}{{\mathfrak n}}
\newcommand{\mol}{{\mathfrak n}}
\newcommand{\mfg}{{\mathfrak g}}

\newcommand{\Ree}{\text{\rm Re}}
\newcommand{\Imm}{\text{\rm Im}}
\newcommand{\III}{I}
\newcommand{\half}{\frac{1}{2}}
\newcommand{\quarter}{\frac{1}{4}}
\newcommand{\dspace}{{{\mathbb C}}_{h}^{\circ}}
\newcommand{\gspace}{{\mathbb C}_h}
\newcommand{\rspace}{{\mathbb R}_h}
\newcommand{\rspacez}{{\mathbb R}_h^{\circ}}
\newcommand{\SSpace}{{\sf S}_h^{\ssy\cset}}
\newcommand{\SSpacer}{{\sf S}_h^{\ssy\rset}}
\newcommand{\ddelta}{\delta_{\star}}

\newcommand{\PP}{{\overset{\ssy\circ}{\sf I}}_h}
\newcommand{\emid}{e_{\ssy\rm mid}}
\begin{document}
\title[]
{Error estimation of the\\
Relaxation Finite Difference Scheme \\
for the nonlinear Schr{\"o}dinger Equation}
\author[]{Georgios E. Zouraris$^{\ddag}$}
\thanks
{$^{\ddag}$ Department of Mathematics and Applied Mathematics,
Division of Applied Mathematics: Differential Equations and Numerical Analysis,
University of Crete, GR-700 13 Voutes Campus, Heraklion, Crete, Greece.}
%
%
%
%
\begin{abstract}
%
%
We consider an initial- and boundary- value problem for the
nonlinear Schr{\"o}dinger equation with homogeneous Dirichlet
boundary conditions in the one space dimension case. We discretize
the problem in space by a central finite difference method
and in time by the Relaxation Scheme proposed by C. Besse
[C. R. Acad. Sci. Paris S{\'e}r. I {\bf 326} (1998), 1427-1432].
We provide optimal order error estimates, in the discrete
$L_t^{\infty}(H_x^1)$ norm, for the approximation error at the time
nodes and at the intermediate time nodes.
In the context of the non linear Schr{\"o}dinger
equation, it is the first time that the derivation of an error
estimate, for a fully discrete method based on the Relaxation Scheme,
is completely addressed.
\end{abstract}
\maketitle
%
%
%
%
%
%
%
%
\section{Introduction}
\subsection{Formulation of the problem}
Let $T>0$,  $\III:=[x_a,x_b]$ be a bounded closed
interval in ${\mathbb R}$, $\QQ:=[0,T]\times I$ and
$\uu:\QQ\mapsto\cset$
be the solution of the following initial and
boundary value problem:
\begin{gather}
\uu_t=\iii\,\uu_{xx}+\iii\,g(|\uu|^2)\,\uu+f
\quad\text{\rm on}\ \ \QQ,
\label{PSL_a}\\
\uu(t,x_a)=u(t,x_b)=0\quad\forall\,t\in[0,T],
\label{PSL_b}\\
\uu(0,x)=\uu_0(x)\quad\forall\,x\in\III,
\label{PSL_c}
\end{gather}
where $g\in C([0,+\infty),\rset)$,
$f\in C(\QQ,\cset)$ and $\uu_0\in C(\III,\cset)$ with
\begin{equation}\label{PSL_d}
u_0(x_a)=u_0(x_b)=0.
\end{equation}
\par
In addition, we assume that the problem above admits a unique solution
$u\in C_{t,x}^{1,2}(Q)$ and that the data $g$, $f$ and $u_0$ are smooth
enough and compatible in order to ensure that the solution and
its higher derivatives are sufficiently, for our purposes, smooth on $Q$.
%
%
%
\subsection{The Relaxation Finite Difference method}
%
%
Let $\nset$ be the set of all positive integers and ${\sf L}:=x_b-x_a$. For given
$N\in\nset$, we define a uniform partition of the time interval $[0,T]$ with
time-step $\tau:=\tfrac{T}{N}$, nodes $t_n:=n\,\tau$ for $n=0,\dots,N$,
and intermediate nodes $t^{n+\half}=t_n+\tfrac{\tau}{2}$ for
$n=0,\dots,N-1$. Also, for given $J\in\nset$, we consider a uniform partition
of $\III$ with mesh-width $h:=\tfrac{{\sf L}}{J+1}$ and nodes 
$x_j:=x_a+j\,h$ for $j=0,\dots,J+1$. Then, we introduce the discrete spaces
\begin{gather*}
\gspace:=\left\{\,(v_j)_{j=0}^{\ssy J+1}:\,\, v_j\in\cset,\,\, j=0,\dots,J+1\,\right\},
\quad
{\dspace}:=\left\{\,(v_j)_{j=0}^{\ssy J+1}\in\gspace:\,\, v_0=v_{\ssy J+1}=0\,\right\},\\
\rspace:=\left\{\,(v_j)_{j=0}^{\ssy J+1}:\,\, v_j\in\rset,\,\, j=0,\dots,J+1\,\right\},
\quad
\rspacez:=\dspace\cap\rspace,
\end{gather*}
a discrete product operator
$\cdot\otimes\cdot:\gspace\times\gspace\mapsto\gspace$ by
\begin{equation*}
(v{\otimes}w)_j=v_j\,w_j,\quad j=0,\dots,J+1,
\quad\forall\,v,w\in\gspace,
\end{equation*}
a discrete Laplacian operator $\Delta_h:\dspace\mapsto\dspace$ by
\begin{equation*}
\Delta_hv_j:=\tfrac{v_{j-1}-2v_j+v_{j+1}}{h^2},
\quad j=1,\dots,J,\quad\forall\,v\in\dspace,
\end{equation*}
In addition, we introduce o\-pe\-rators
${\sf I}_h:C(\III,{\mathbb C})\mapsto\gspace$ and
$\PP:C(\III,{\mathbb C})\mapsto\dspace$, which, for
given $z\in C(\III,\cset)$, are defined, respectively,
by $({\sf I}_h z)_j:=z(x_j)$ for $j=0,\dots,J+1$ and
$(\PP z)_j:=z(x_j)$ for $j=1,\dots,J$, 
and a Discrete Elliptic Projection operator
${\sf R}_h: C^2(I;\cset)\mapsto\dspace$ (cf. \cite{AD1991})
by requiring
\begin{equation}\label{ELP_1}
\Delta_h({\sf R}_h[v])=\PP(v'')\quad\forall\,v\in C^2(I;\cset).
\end{equation}
Finally, for $\ell\in\nset$ and for any function
$q:\cset^{\,\ell}\mapsto\cset$ and any
$w=(w^1,\dots,w^{\ell})\in(\gspace)^{\ell}$, we define
$q(w)\in\gspace$ by
$(q(w))_j:=g\left(w_j^1,\dots,w^{\ell}_j\right)$
for $j=0,\dots,J+1$.
%
%
%
\par
The Relaxation Finite Difference (RFD) method combines a standard finite
difference method for space discretization, with the Besse Relaxation
Scheme for time-stepping (cf. Section 5 in \cite{Besse2}).
Its algorithm consists of the following steps:
\par\noindent\vskip0.2truecm\par
{\tt Step 1}: Define $\UUU^0\in\dspace$ by
\begin{equation}\label{BRS_1}
\UUU^0:={\sf R}_h[\uu_0]
\end{equation}
and find $\UUU^{\half}\in\dspace$ such that
\begin{equation}\label{BRS_12}
\UUU^{\half}-\UUU^0
=\iii\,\tfrac{\tau}{2}\,\Delta_{h}\left(\,\tfrac{\UUU^{\half}+\UUU^0}{2}\,\right)
+\iii\,\tfrac{\tau}{2}\,g\big(|\uu^0|^2\big)
\otimes\left(\tfrac{\UUU^{\half}+\UUU^0}{2}\right)
+\tfrac{\tau}{2}\,\PP\left[\tfrac{f(t^{\frac{1}{2}},\cdot)+f(t_0,\cdot)}{2}\right].
\end{equation}
\par\noindent\vskip0.2truecm\par
{\tt Step 2}:
Define $\Phi^{\half}\in\rspace$ by
\begin{equation}\label{BRS_13}
\Phi^{\half}:=g(|\UUU^{\half}|^2)
\end{equation}
and find $\UUU^1\in\dspace$ such that
\begin{equation}\label{BRS_2}
\UUU^1-\UUU^0=\iii\,\tau
\,\Delta_{h}\left(\,\tfrac{\UUU^1+\UUU^0}{2}\,\right)
+\iii\,\tau\,\Phi^{\half}\otimes\left(\tfrac{\UUU^1+\UUU^0}{2}\right)
+\tau\,\PP\left[\tfrac{f(t_1,\cdot)+f(t_0,\cdot)}{2}\right].
\end{equation}
\par\noindent\vskip0.2truecm\par
{\tt Step 3}: For $n=1,\dots,N-1$, first define $\Phi^{n+\half}\in\rspace$ by
\begin{equation*}\label{BRS_3}
\Phi^{n+\half}:=2\,g(|\UUU^n|^2)-\Phi^{n-\half}
\end{equation*}
and then find $\UUU^{n+1}\in\dspace$ such that
\begin{equation}\label{BRS_4}
\UUU^{n+1}-\UUU^{n}=\iii\,\tau\,\Delta_h\left(\tfrac{\UUU^{n+1}+\UUU^{n}}{2}\right)
+\iii\,\tau\,\Phi^{n+\half}\otimes\left(\tfrac{\UUU^{n+1}+\UUU^{n}}{2}\right)
+\tau\,\PP\left[\tfrac{f(t_{n+1},\cdot)+f(t_n,\cdot)}{2}\right].
\end{equation}
\begin{remark}
It is easily verified that the (RFD) method requires the numerical solution of
a linear tridiagonal system of algebraic equation at every time step.
\end{remark}
\begin{remark}
The (RFD) approximations are, unconditionally,
well-posed (see Lemma~\ref{Plan_1}).
\end{remark}
\begin{remark}
In \eqref{BRS_13}, we construct a second order approximation
$\Phi^{\half}$ of $g(|u(t^{\half},\cdot)|^2)$, instead of the
the first order approximation $g(|{\sf I}_h[u_0]|^2)$ proposed
in \cite{Besse2}. Later, we show that both choices yield a second
order convergence of the numerical method at the time nodes
(see Theorems  \ref{DR_Final} and \ref{CNV_Final_2}).
\end{remark}
%
%
%
\subsection{Related references and main results}
It is well known (see, e.g. \cite{Akrivis1}, \cite{KarAD}, \cite{ADKar},
\cite{KarMak}) that the application of an implicit time-stepping method
to a non linear Schr{\"o}dinger equation,
creates, at every time step, the need of approximating the solution to a
system of complex, non linear algebraic equations via an iterative solver.
One way to keep the computational complexity of the method
in a predefine level, is to employ a linear implicit
time-stepping method (see, e.g., \cite{Sulem}, \cite{Fei},
\cite{Georgios1}, \cite{Wang}) that handles the linear part of the equation implicitly
and the non linear part of the equation explicitly or semi-implicitly
and thus require, at every time step, the solution of a linear system
of algebraic equations.
Within this context, C. Besse \cite{Besse1} proposed
the {\it Relaxation Scheme} (RS) which was a new,
linear implicit, time-stepping method, conserving in
a discrete way the charge and the energy.  The Relaxation Scheme (RS)
along with a finite element or a finite difference space discretization,
is computationally efficient (see, e.g., \cite{ABB2013}, \cite{Mitsotakis},
\cite{Patrick}) and performs as a second order method (see, e.g.,
\cite{Besse2}, \cite{Mitsotakis}).
Later, C. Besse \cite{Besse2}, analysing the (RS) as
a semidiscrete in time method approximating the solution of the
Cauchy problem for the nonlinear Schr{\"o}dinger equation with
power non-linearity, show its convergence for small final time $T$,
without concluding a convergent rate with respect to the time step.
C. Besse et al. \cite{Besse3} focusing on the
cubic nonlinear Schr{\"o}dinger equation, combine the
stability results in \cite{Besse2} with a proper consistency
argument based on the Taylor formula and bound the approximation error of
the (RS) in the $H^s(\rset^d)-$norm by the error approximating 
$g(|u(t^{\half},\cdot)|^2)$ coupled with a,
second order with respect to the time-step, additive term.
The error estimate obtained yields a first order convergence of
the (RS) when  $g(|u(t^{\half},\cdot)|^2)$ is approximated by $g(|u_0|^2)$
(see, e.g. \cite{Besse1}, \cite{Besse2}),
and thus it is not able to explain its second order convergence
that has been observed experimentally by several authors
(see, e.g., \cite{Besse2}, \cite{Mitsotakis}).
Also, the error analysis developed in \cite{Besse2} and \cite{Besse3}
is based on the derivation of a priori bounds for the time-discrete
approximations in the $H^{s+2}(\rset^d)-$norm with $s>\tfrac{d}{2}$,
(cf. Hypotheses 2 in \cite{Besse2}).
However, this approach can not be adopted for the analysis of
fully discrete methods, where the (RS) is coupled with a finite
difference or a finite element method for space discretization.
The problem is coming from the fact that on one hand
the finite element approximations are, usually, only $H^1$ functions,
and on the other hand the finite difference approximations are not
able to mimic, in a discrete way, all the compatibility conditions
that the solution to the continuous problem satisfy. This indicates
that the error estimation in the fully discrete case has to follow
a different path.
\par 
Recently, considering the approximation
of a semilinear heat equation in the one space dimension by the (RS) coupled with
a central finite difference method,
we provided a second order error estimate using energy-type techniques
\cite{Georgios2}.
Unfortunately, the latter convergence analysis
can not be extended in the case of the nonlinear Schr{\"o}dinger equation,
because it is based on the stability properties of the parabolic problems.
%
%
\par
In the work at hands, our aim is to contribute to the understanding of the
convergence nature of the (RS) by investigating the convergence of the
(RFD) method formulated in \eqref{BRS_1}-\eqref{BRS_4}
for the approximation of the solution to the problem \eqref{PSL_a}-\eqref{PSL_d}.
By building up a proper stability argument based on \cite{Besse2}
and formulating a proper modified version of the numerical method
based on the framework proposed in \cite{Georgios1}, we are able to
prove a new, optimal, second order error estimate in a discrete
$L_t^{\infty}(H_x^1)-$norm at the nodes and the intermediate time nodes,
without restrictions on the final time $T$ and avoiding to impose 
coupling conditions on the mesh parameters. Also, considering
the first order in time approximation $\Phi^{\half}=g(|{\sf I}_h[u_0]|^2)$,
we show  that convergence of the method is still second order at the nodes
while remain first order in time at the intermediate nodes.
\par
We close this section by giving a brief overview of the paper.
In Section~\ref{Section2} we introduce notation and we prove
a series of auxiliary results that we will often use later
in the analysis of the numerical method.
Section~\ref{Section3} is dedicated to the definition and estimations
of the consistency errors allong with the presentation of the approximation
properties of the Discrete Elliptic Projection operator.
Finally, Section~\ref{Section4} contains the convergence analysis of the
(RFD) method via the construction and the analysis of a modified scheme.
%
%
%
%
%
\section{Preliminaries}\label{Section2}
%
%
In this section, we introduce additional notation and present a
series of basic auxiliary results that we will use often in
the convergence analysis of the numerical method.
\subsection{Additional notation}
Let us define the discrete space
$\SSpace:=\left\{\,(z_j)_{j=0}^{\ssy J}:\,\,z_j\in\cset,\,\, j=0,\dots,J\right\}$
along with its real subset
$\SSpacer:=\left\{\,(z_j)_{j=0}^{\ssy J}:\,\,z_j\in\rset,\,\, j=0,\dots,J\right\}$,
and we define the discrete space derivative operator
$\delta_h:\gspace\mapsto\SSpace$ by
$\delta_hv_j:=\tfrac{v_{j+1}-v_j}{h}$ for $j=0,\dots,J$ and $v\in\gspace$.
On $\SSpace$ we define the inner product $(\!\!(\cdot,\cdot)\!\!)_{0,h}$ by
$(\!\!(z,v)\!\!)_{0,h}:=h\sum_{j=0}^{\ssy J}z_j\,{\overline v_j}$ for $z,v\in\SSpace$,
and we will denote by $|\!|\!|\cdot|\!|\!|_{0,h}$ the corresponding norm,
i.e. $|\!|\!|z|\!|\!|_{0,h}:=\sqrt{(\!\!(z,z)\!\!)_{0,h}}$ for $z\in\SSpace$.
Also, we define a discrete maximum norm on $\SSpace$
by $|\!|\!|v|\!|\!|_{\infty,h}:=\max_{0\leq{j}\leq{\ssy J}}|v_j|$ for $v\in\SSpace$.
\par
We provide $\gspace$ with the discrete inner product
$(\cdot,\cdot)_{0,h}$ by $(v,z)_{0,h}:=h\sum_{j=0}^{\ssy J+1}v_j\,{\overline z_j}$
for $v,z\in\gspace$, and we shall denote by $\|\cdot\|_{0,h}$ its
induced norm, i.e.
$\|v\|_{0,h}:=\sqrt{(v,v)_{0,h}}$ for $v\in\gspace$.
Also, we define on $\gspace$
a discrete maximum norm $|\cdot|_{\infty,h}$
by $|w|_{\infty,h}:=\max_{0\leq j\leq{J+1}}|w_j|$ for
$w\in\gspace$, a discrete $H^1-$seminorm $|\cdot|_{1,h}$
by $|w|_{1,h}:=|\!|\!|\delta_hw|\!|\!|_{0,h}$ for $w\in\gspace$
and a discrete $H^1-$norm $\|\cdot\|_{1,h}$ by
$\|w\|_{1,h}:=(\|w\|_{0,h}^2+|w|_{1,h}^2)^{\half}$ for $w\in\dspace$.
\subsection{Auxiliary results}
\par
It is easily seen that, for $v\in\dspace$,
the following inequalities hold
\begin{gather}
|v|_{\infty,h}\leq\,\sqrt{{\sf L}}\,|v|_{1,h}\label{LH1}\\
\|v\|_{0,h}\leq\,{\sf L}\,|v|_{1,h}\label{dpoincare}\\
|v|_{1,h}\leq\,2\,h^{-1}\,\|v\|_{0,h}.\label{H1L2}
\end{gather}
Under the light of \eqref{dpoincare}, the seminorm
$|\cdot|_{1,h}$ is a norm on $\dspace$
which is equivalent to $\|\cdot\|_{1,h}$.
%
%
%
%
%
\begin{lem}\label{LemmaX1}
For all $v,z\in\dspace$ it holds that
\begin{gather}
(\Delta_hv,z)_{0,h}=-(\!\!(\delta_hv,\delta_hz)\!\!)_{0,h}=(v,\Delta_hz)_{0,h},
\label{NewEra1}\\
(\Delta_hv,v)_{0,h}=-|v|^2_{1,h}.\label{NewEra2}
\end{gather}
\end{lem}
%
%
%
%
\begin{proof}
Let $v,z\in\dspace$. First, we obtain \eqref{NewEra1} 
proceeding as follows
\begin{equation*}
(\Delta_hv,z)_{0,h}=\sum_{j=1}^{\ssy J}\left[(\delta_hv)_{j}-(\delta_hv)_{j-1}\right]
\,{\overline z_j}
=\sum_{j=0}^{\ssy J}(\delta_hv)_{j}\,{\overline z_j}
-\sum_{j=0}^{\ssy J}(\delta_hv)_j\,{\overline z_{j+1}}
=-(\!\!(\delta_hv,\delta_hz)\!\!)_{0,h}.
\end{equation*}
Then. we observe that \eqref{NewEra2} is a simple
consequence of \eqref{NewEra1}.
\end{proof}
%
%
%
%
%
\begin{lem}\label{LemmaX2}
Let $\varepsilon>0$, ${\mathfrak g}\in C^2(\rset;\rset)$,
$\mfg'_{\varepsilon}:=\sup_{|x|\in[0,\varepsilon]}|\mfg'(x)|$,
$\mfg''_{\varepsilon}:=\sup_{|x|\in[0,\varepsilon]}|\mfg''(x)|$
and
${\mathfrak R}_h^{\varepsilon}:=\{v\in\rspace:\,\,|v|_{\infty,h}\leq\varepsilon\}$.
Then, for $v,w\in{\mathfrak R}_h^{\varepsilon}$, it holds that
\begin{equation}\label{BasicHot1}
\|\mfg(v)-\mfg(w)\|_{0,h}\leq\mfg'_{\varepsilon}\,\|v-w\|_{0,h}
\end{equation}
and
\begin{equation}\label{BasicHot2}
|\mfg(v)-\mfg(w)|_{1,h}\leq
\mfg'_{\varepsilon}\,|v-w|_{1,h}
+\mfg''_{\varepsilon}
\,|\!|\!|\delta_hw|\!|\!|_{\infty,h}\,\|v-w\|_{0,h}.
\end{equation}
\end{lem}
%
%
%
%
%
\begin{proof}
Let $v,w\in{\mathfrak R}_h^{\varepsilon}$. Then, for $s\in[0,1]$,
we define ${\mathfrak c}^s\in{\mathfrak R}_h^{\varepsilon}$
by ${\mathfrak c}^s:=s\,v+(1-s)\,w$
and ${\mathfrak a}^s$, ${\mathfrak b}^s\in\SSpacer$
by ${\mathfrak a}^s_j:=s\,v_{j+1}+(1-s)\,v_j$ and
${\mathfrak b}^s_j:=s\,w_{j+1}+(1-s)\,w_j$
for $j=0,\dots,J$.
Applying the mean value theorem, we have
$\mfg(v)-\mfg(w)=(v-w)\otimes\left(\int_0^1\mfg'({\mathfrak c}^s)\;ds\right)$,
which, easily, yields \eqref{BasicHot1}.
Applying, again, the mean value theorem, we conclude that
\begin{equation*}
\begin{split}
|\delta_h(\mfg(v)-\mfg(w))_j|=&\,\left|(\delta_h(v-w))_j
\,\left(\int_0^1\mfg'({\mathfrak a}^s_j)\;ds\right)
+(\delta_hw)_j\,\left(\int_0^1\left[\mfg'({\mathfrak a}^s_j)
-\mfg'({\mathfrak b}^s_j)\right]\;ds\right)\right|\\
\leq&\,\mfg''_{\varepsilon}\,
|(\delta_hw)_j|\,\int_0^1\left|\,s(v_{j+1}-w_{j+1})+(1-s)\,(v_j-w_j)\,\right|\;ds
+\mfg'_{\varepsilon}\,|(\delta_h(v-w))_j|\\
\leq&\,\tfrac{1}{2}\,\mfg''_{\varepsilon}\,|(\delta_hw)_j|\,
\left(|v_{j+1}-w_{j+1}|+|v_j-w_j|\right)
+\mfg'_{\varepsilon}\,|(\delta_h(v-w))_j|,\quad j=0,\dots,J,\\
\end{split}
\end{equation*}
which, easily, yields \eqref{BasicHot2}.
\end{proof}
%
%
%
%
\begin{lem}\label{LemmaX3}
Let $\varepsilon>0$,
${\mathfrak R}_h^{\varepsilon}:=\{v\in\rspace:\,\,|v|_{\infty,h}\leq\varepsilon\}$,
$\mfg\in C^3(\rset;\rset)$, 
$\mfg'_{\varepsilon}:=\sup_{\ssy |x|\in[0,\varepsilon]}|\mfg'(x)|$,
$\mfg''_{\varepsilon}:=\sup_{\ssy |x|\in[0,\varepsilon]}|\mfg''(x)|$,
$\mfg'''_{\varepsilon}:=\sup_{\ssy |x|\in[0,\varepsilon]}|\mfg'''(x)|$.
Then, for $v^a,v^b,z^a,z^b\in{\mathfrak R}_h^{\varepsilon}$, it holds that
\begin{equation}\label{BasicHot3}
\begin{split}
\|\mfg(v^a)-\mfg(v^b)-\mfg(z^a)+\mfg(z^b)\|_{0,h}
\leq&\,\mfg''_{\varepsilon}\,\,|z^a-z^b|_{\infty,h}\,\|v^b-z^b\|_{0,h}\\
&+\left(\,\mfg'_{\varepsilon}
+\mfg''_{\varepsilon}\,\,|z^a-z^b|_{\infty,h}\,\right)
\,\|v^a-v^b-z^a+z^b\|_{0,h}\\
\end{split}
\end{equation}
and
\begin{equation}\label{BasicHot4}
\begin{split}
|\mfg(v^a)-\mfg(v^b)-\mfg(z^a)+\mfg(z^b)|_{1,h}
\leq&\,(\mfg'_{\varepsilon}
+\mfg''_{\varepsilon}\,|z^a-z^b|_{\infty,h})
\,|v^a-v^b-z^a+z^b|_{1,h}\\
&\,+\tfrac{1}{2}\,\mfg''_{\varepsilon}
\,(|v^a|_{1,h}+|v^b|_{1,h})\,|v^a-v^b-z^a+z^b|_{\infty,h}\\
&\,+\mfg''_{\varepsilon}\,|z^a-z^b|_{\infty,h}\,|v^b-z^b|_{1,h}\\
&\,+{\mathcal F}_{\varepsilon}(z^a,z^b)\,
\,\left(\,\|v^a-v^b-z^a+z^b\|_{0,h}+\|v^b-z^b\|_{0,h}\,\right),\\
\end{split}
\end{equation}
where
\begin{equation*}
\begin{split}
{\mathcal F}_{\varepsilon}(z^a,z^b):=&\,
\mfg''_{\varepsilon}\,|\!|\!|\delta_h(z^a-z^b)|\!|\!|_{\infty,h}
+\mfg'''_{\varepsilon}\,|z^a-z^b|_{\infty,h}
\,\left[\,|\!|\!|\delta_h(z^a-z^b)|\!|\!|_{\infty,h}
+|\!|\!|\delta_hz^b|\!|\!|_{\infty,h}\,\right].\\
\end{split}
\end{equation*}
\end{lem}
%
%
%
%
\begin{proof}
Let $v^a,v^b,z^a,z^b\in{\mathfrak R}_h^{\varepsilon}$. We simplify
the notation, first, by defining
${\mathfrak a}^s$, ${\mathfrak b}^s\in{\mathfrak R}_h^{\varepsilon}$ by
${\mathfrak a}^s:=s\,v^a+(1-s)\,v^b$ and ${\mathfrak b}^s:=s\,z^a+(1-s)\,z^b$
for $s\in[0,1]$, and then, by introducing ${\sf f}^{\ssy A},{\sf f}^{\ssy B}\in\rspace$
by ${\sf f}^{\ssy A}:=\int_0^1\mfg'({\mathfrak a}^s)\;ds$
and
${\sf f}^{\ssy B}:=\int_0^1\left[{\mathfrak g}'({\mathfrak a}^s)
-{\mathfrak g}'({\mathfrak b}^s)\right]\;ds$. Also, we set $e^a:=v^a-z^a$ and
$e^b:=v^b-z^b$.
%
%
\par\noindent\vskip0.3truecm\par\noindent
\boxed{\sf Part\,\,\, I.} First, we use the definition of
${\sf f}^{\ssy A}$ and the mean value theorem, to get
\begin{equation}\label{OPAP_2}
|{\sf f}^{\ssy A}|_{\infty,h}\leq\,\mfg'_{\varepsilon}
\end{equation}
and
\begin{equation*}
\begin{split}
\left|(\delta_h{\sf f}^{\ssy A})_j\right|\leq&\,\tfrac{1}{h}\,\int_0^1|\mfg'({\mathfrak a}^s_{j+1})
-\mfg'({\mathfrak a}^s_j)|\;ds\\
\leq&\,\mfg''_{\varepsilon}\,
\int_0^1\left|s\,\delta_hv^a_j+(1-s)\,\delta_hv^b_j\,\right|\;ds\\
\leq&\,\tfrac{1}{2}\,\mfg''_{\varepsilon}
\,\left(\,|\delta_hv^a_j|+|\delta_hv^b_j|\,\right),
\quad j=0,\dots,J,\\
\end{split}
\end{equation*}
which, obviously, yields
\begin{equation}\label{OPAP_3}
|{\sf f}^{\ssy A}|_{1,h}\leq\,\tfrac{1}{2}\,\mfg''_{\varepsilon}
\,\left(|v^a|_{1,h}+|v^b|_{1,h}\right).
\end{equation}
Next, we use the definition of ${\sf f}^{\ssy B}$ and the mean value theorem,
to obtain
\begin{equation*}
\begin{split}
|{\sf f}^{\ssy B}_j|\leq&\,\mfg''_{\varepsilon}\,\,\int_0^1|{\mathfrak a}_j^s
-{\mathfrak b}_j^s|\;ds\\
\leq&\,\mfg''_{\varepsilon}\,\,
\int_0^1|s\,(v^a-v^b-z^a+z^b)_j
+(v^b-z^b)_j|\;ds\\
\leq&\,\mfg''_{\varepsilon}\,\,\left(\,|(v^a-v^b-z^a+z^b)_j|
+|(v^b-z^b)_j|\,\right),\quad j=0,\dots,J+1,\\
\end{split}
\end{equation*}
which, leads to
\begin{equation}\label{OPAP_4}
\|{\sf f}^{\ssy B}\|_{0,h}\leq\,\mfg''_{\varepsilon}\,
\left(\,\|e^a-e^b\|_{0,h}+\|e^b\|_{0,h}\,\right).
\end{equation}
Also, for $s\in[0,1]$, we apply \eqref{BasicHot2} to arrive at
\begin{equation*}\label{OPAP_5}
\begin{split}
|{\mathfrak g}'({\mathfrak a}^s)-{\mathfrak g}'({\mathfrak b}^s)|_{1,h}
\leq&\,\mfg''_{\varepsilon}\,|{\mathfrak a}^s-{\mathfrak b}^s|_{1,h}
+{\mathfrak g}'''_{\varepsilon}
\,|\!|\!|\delta_h{\mathfrak b}^s|\!|\!|_{\infty,h}\,\|{\mathfrak a}^s-{\mathfrak b}^s\|_{0,h}\\
\leq&\,\mfg''_{\varepsilon}\,\,\left(s\,\,|e^a-e^b|_{1,h}+\,|e^b|_{1,h}\,\right)\\
&\quad+{\mathfrak g}'''_{\varepsilon}
\,(s\,|\!|\!|\delta_h(z^a-z^b)|\!|\!|_{\infty,h}
+\,|\!|\!|\delta_hz^b|\!|\!|_{\infty,h})\,\big]
\,\left(s\,\,\|e^a-e^b\|_{0,h}+\,\|e^b\|_{0,h}\,\right),\\
\end{split}
\end{equation*}
which we use to obtain
\begin{equation}\label{OPAP_6}
\begin{split}
|{\sf f}^{\ssy B}|_{1,h}\leq&\int_0^1|{\mathfrak g}'({\mathfrak a}^s)
-{\mathfrak g}'({\mathfrak b}^s)|_{1,h}\;ds\\
\leq&\,\mfg''_{\varepsilon}\,(|e^a-e^b|_{1,h}+|e^b|_{1,h})\\
&\quad+{\mathfrak g}'''_{\varepsilon}
\,\left(|\!|\!|\delta_h(z^a-z^b)|\!|\!|_{\infty,h}
+|\!|\!|\delta_hz^b|\!|\!|_{\infty,h}\right)
\,\left(\,\|e^a-e^b\|_{0,h}+\|e^b\|_{0,h}\,\right).\\
\end{split}
\end{equation}
%
%
%
\par\noindent\vskip0.3truecm\par\noindent
\boxed{\sf Part\,\,\, II.} Using the mean value theorem, we obtain
\begin{equation}\label{OPAP_7}
{\mathfrak g}(v^a)-{\mathfrak g}(v^b)-{\mathfrak g}(z^a)+{\mathfrak g}(z^b)
={\mathfrak L}^{\ssy A}+{\mathfrak L}^{\ssy B},
\end{equation}
where ${\mathfrak L}^{\ssy A}$, ${\mathfrak L}^{\ssy B}\in\rspace$ are defined by
${\mathfrak L}^{\ssy A}:=(v^a-v^b-z^a+z^b)\otimes{\sf f}^{\ssy A}$ and
${\mathfrak L}^{\ssy B}:=(z^a-z^b)\otimes{\sf f}^{\ssy B}$. Thus, using
\eqref{OPAP_2} and \eqref{OPAP_4}, we have
\begin{equation}\label{OPAP_8}\
\begin{split}
\|{\mathfrak L}^{\ssy A}\|_{0,h}\leq&\,\mfg'_{\varepsilon}\,\|e^a-e^b\|_{0,h},\\
\|{\mathfrak L}^{\ssy B}\|_{0,h}\leq&\,\mfg''_{\varepsilon}
\,|z^a-z^b|_{\infty,h}\,\left(\|e^a-e^b\|_{0,h}+\|e^b\|_{0,h}\right).\\
\end{split}
\end{equation}
Thus, \eqref{BasicHot3} follows, easily, from \eqref{OPAP_7} and \eqref{OPAP_8}.
%
%
%
%
\par\noindent\vskip0.3truecm\par\noindent
\boxed{\sf Part\,\,\, III.} Observing that
\begin{equation*}
\begin{split}
(\delta_h{\mathfrak L}^{\ssy A})_{j}=&
\,{\sf f}^{\ssy A}_{j+1}\,\delta_h(v^a-v^b-z^a+z^b)_j
+(\delta_h{\sf f}^{\ssy A})_j\,(v^a-v^b-z^a+z^b)_j,\\
(\delta_h{\mathfrak L}^{\ssy B})_{j}=&\,{\sf f}^{\ssy B}_{j+1}\,\delta_h(z^a-z^b)_j
+(\delta_h{\sf f}^{\ssy B})_j\,(z^a-z^b)_j\\
\end{split}
\end{equation*}
for $j=0,\dots,J$, we obtain
\begin{equation}\label{OPAP_9}
\begin{split}
|{\mathfrak L}^{\ssy A}|_{1,h}\leq&\,|{\sf f}^{\ssy A}|_{\infty,h}\,|e^a-e^b|_{1,h}
+|{\sf f}^{\ssy A}|_{1,h}\,|e^a-e^b|_{\infty,h},\\
|{\mathfrak L}^{\ssy B}|_{1,h}\leq&\,|\!|\!|\delta_h(z^a-z^b)|\!|\!|_{\infty,h}
\,\|{\sf f}^{\ssy B}\|_{0,h}
+|z^a-z^b|_{\infty,h}\,|{\sf f}^{\ssy B}|_{1,h}.\\
\end{split}
\end{equation}
Using \eqref{OPAP_9}, \eqref{OPAP_2} and \eqref{OPAP_3}, we have
\begin{equation}\label{OPAP_10}
\begin{split}
|{\mathfrak L}^{\ssy A}|_{1,h}\leq&\,
{\mathfrak g}'_{\varepsilon}\,|e^a-e^b|_{1,h}
+\tfrac{1}{2}\,{\mathfrak g}''_{\varepsilon}
\,(\,|v^a|_{1,h}+|v^b|_{1,h}\,)\,|e^a-e^b|_{\infty,h}.
\end{split}
\end{equation}
Combining \eqref{OPAP_9}, \eqref{OPAP_4} and \eqref{OPAP_6}, we arrive at
\begin{equation}\label{OPAP_11}
\begin{split}
|{\mathfrak L}^{\ssy B}|_{1,h}\leq&\,\mfg''_{\varepsilon}
\,|\!|\!|\delta_h(z^a-z^b)|\!|\!|_{\infty,h}\,
\left(\,\|e^a-e^b\|_{0,h}+\|e^b\|_{0,h}\,\right)\\
&\,+|z^a-z^b|_{\infty,h}\,{\mathfrak g}'''_{\varepsilon}
\,\left[|\!|\!|\delta_h(z^a-z^b)|\!|\!|_{\infty,h}
+|\!|\!|\delta_hz^b|\!|\!|_{\infty,h}\right]
\,\left(\,\|e^a-e^b\|_{0,h}+\|e^b\|_{0,h}\,\right)\\
&\,+|z^a-z^b|_{\infty,h}\,\mfg''_{\varepsilon}
\,\left(\,|e^a-e^b|_{1,h}+|e^b|_{1,h}\,\right).\\
\end{split}
\end{equation}
Finally, \eqref{BasicHot4} follows, easily, in view of \eqref{OPAP_7}, \eqref{OPAP_10}
and \eqref{OPAP_11}.
\end{proof}
%
%
%
%
\begin{lem}\label{LemmaX4}
%
For $v^a,v^b,z^a,z^b\in\gspace$, it holds that
\begin{equation}\label{Jalapeno1}
\big\||v^a|^2-|z^a|^2\big\|_{0,h}
\leq\,(|v^a|_{\infty,h}+|z^a|_{\infty,h})\,\|v^a-z^a\|_{0,h},
\end{equation}
\begin{equation}\label{Jalapeno12}
\begin{split}
\left|\,|v^a|^2-|z^a|^2\,\right|_{1,h}
\leq&\,2\,|v^a|_{\infty,h}\,|v^a-z^a|_{1,h}
+2\,|\!|\!|\delta_h(z^a)|\!|\!|_{\infty,h}\,\|v^a-z^a\|_{0,h},\\
\end{split}
\end{equation}
\begin{equation}\label{Jalapeno2}
\begin{split}
\big\|\,|v^a|^2-|v^b|^2-|z^a|^2+|z^b|^2\,\big\|_{0,h}
\leq&\,2\,|z^a-z^b|_{\infty,h}\,\|v^b-z^b\|_{0,h}\\
&+(|v^a|_{\infty,h}+|v^b|_{\infty,h}+|z^a-z^b|_{\infty,h})
\,\|v^a-v^b-z^a+z^b\|_{0,h},\\
\end{split}
\end{equation}
and
\begin{equation}\label{Jalapeno21}
\begin{split}
\big||v^a|^2-|v^b|^2-|z^a|^2+|z^b|^2\big|_{1,h}
\leq&\,G^{\ssy A}(v^a,v^b,z^a,z^b)\,|v^a-v^b-z^a-z^b|_{1,h}\\
&\,+G^{\ssy B}(v^a,v^b,z^a,z^b)
\,|v^a-v^b-z^a-z^b|_{\infty,h}\\
&\,+2\,|z^a-z^b|_{\infty,h}\,|v^b-z^b|_{1,h}
+2\,|\!|\!|\delta_h(z^a-z^b)|\!|\!|_{\infty,h}\,\|v^b-z^b\|_{0,h}\\
\end{split}
\end{equation}
where
\begin{equation*}
\begin{split}
G^{\ssy A}(v^a,v^b,z^a,z^b):=&\,|v^a|_{\infty,h}+|v^b|_{\infty,h}+|z^a-z^b|_{\infty,h},\\
G^{\ssy B}(v^a,v^b,z^a,z^b):=&\,|v^a|_{1,h}
+|v^b|_{1,h}+|z^a-z^b|_{1,h}.\\
\end{split}
\end{equation*}
\end{lem}
%
%
%
\begin{proof}
Let  $v^a,v^b,z^a,z^b\in\gspace$ and $\zeta:=v^a-v^b-z^a+z^b$.
The inequalities \eqref{Jalapeno1},
\eqref{Jalapeno12} and \eqref{Jalapeno2} follow easily
by observing that
\begin{equation*}
\begin{split}
(|v^a|^2-|z^a|^2)_j
=&\,\Ree\left[\,{\overline{(v^a+z^a)_j}}\,(v^a-z^a)_j\,\right],\\
(|v^a|^2-|v^b|^2-|z^a|^2+|z^b|^2)_j=&\,\Ree\left[\,
{\overline{(v^a+v^b)_j}}\,\zeta_j
+{\overline{(z^a-z^b)_j}}\,\zeta_j
+2\,{\overline{(z^a-z^b)_j}}\,(v^b-z^b)_j\,\right]\\
\end{split}
\end{equation*}
for $j=0,\dots,J+1$, and
\begin{equation*}
\begin{split}
\delta_h(|v^a|^2-|z^a|^2)_j=&\,\Ree\left[\,{\overline{(v^a_{j+1}+v^a_j)}}
\,\delta_h(v^a-z^a)_{j}\right.\\
&\hskip1.5truecm\left.+{\overline{\delta_h(z^a)_j}}
\,[(v^a-z^a)_{j+1}+(v^a-z^a)_j]\,\right],\\
\delta_h(|v^a|^2-|v^b|^2-|z^a|^2+|z^b|^2)_j=&\,\Ree\left[\,
{\overline{(v^a+v^b)_{j+1}}}\,
\delta_h(\zeta)_j+{\overline{\delta_h(v^a+v^b)_{j}}}\,
\zeta_j\,\right]\\
&\,+\Ree\left[\,
{\overline{(z^a-z^b)_{j+1}}}\,
\delta_h(\zeta)_j+{\overline{\delta_h(z^a-z^b)_{j}}}\,
\zeta_j\,\right]\\
&+2\,\Ree\left[\,
{\overline{(z^a-z^b)_{j+1}}}\,\delta_h(v^b-z^b)_j
+{\overline{\delta_h(z^a-z^b)_{j}}}\,(v^b-z^b)_j\,\right]\\
\end{split}
\end{equation*}
for $j=0,\dots,J$.
\end{proof}
%
%
\subsection{A molifier}
Let $\delta>0$, $p_{\delta}:[\delta,2\delta]\mapsto{\mathbb R}$ be the unique polynomial
of ${\mathbb P}^7[\delta,2\delta]$ satisfying
\begin{equation*}
p_{\delta}(\delta)=\delta,\,\,p'_{\delta}(\delta)=1,\,\,p''_{\delta}(\delta)=p'''_{\delta}(\delta)=0,
\,\,p_{\delta}(2\delta)=2\delta,\,\,p'_{\delta}(2\delta)=p''_{\delta}(2\delta)=p'''_{\delta}(2\delta)=0,
\end{equation*}
and $\gff_{\delta}\in{\sf C}_b^3(\rset;\rset)$ be an odd function 
(cf. \cite{KarMak}, \cite{Georgios1}) defined by
\begin{equation}\label{MLF_1}
\mol_{\delta}(x):=\left\{
\begin{aligned}
&x,\hskip2.40truecm
\mbox{if}\ \ x\in[0,\delta],\\
&p_{\ssy\delta}(x),\hskip1.76truecm
\mbox{if}\ \ x\in (\delta,2\delta],\\
&2\,\delta,\hskip2.22truecm\mbox{if}\ \ x> 2\delta,\\
\end{aligned}
\right.\quad\forall\,x\ge 0.
\end{equation}
%
%
%
Then, we define a complex molifier
$\gamma_{\delta}:\cset\mapsto\cset$ (cf. \cite{Georgios1}) by
\begin{equation}\label{MLF_2}
\gamma_{\delta}(z):=\gff_{\delta}(\Ree(z))+\iii\,\gff_{\delta}(\Imm(z))
\quad\forall\,z\in\cset,
\end{equation}
for which it is, easily, verified that
\begin{equation}\label{MLF_4}
|\gamma_{\delta}(z)|\leq\,\sqrt{2}\,\sup_{\ssy\rset}|\gff_{\delta}|\quad\forall\,z\in\cset,
\end{equation}
\begin{equation}\label{MLF_5}
|\gamma_{\delta}(z_1)-\gamma_{\delta}(z_2)|
\leq\,\sup_{\ssy\rset}|\gff_{\delta}'|\,|z_1-z_2|\quad\forall\,z_1,z_2\in\cset
\end{equation}
and
\begin{equation}\label{MLF_3}
\gamma_{\delta}(z)=z\quad\forall\,z\in\left\{z\in{\mathbb C}:\,\,|z|\leq\delta\right\}.
\end{equation}
%
%
%
%
\begin{lem}\label{LemmaX5}
Let $\delta>0$,
$\mol'_{\delta,\infty}:=\sup_{\ssy\rset}|\mol'_{\delta}|$,
$\mol''_{\delta,\infty}:=\sup_{\ssy\rset}|\mol''_{\delta}|$
and
$\mol'''_{\delta,\infty}:=\sup_{\ssy\rset}|\mol'''_{\delta}|$.
Then, for all $v^a,v^b,z^a,z^b\in\gspace$, it holds that
\begin{equation}\label{BasicHot5}
\|\gamma_{\delta}(v^a)-\gamma_{\delta}(v^b)\|_{0,h}
\leq\,\mol'_{\delta,\infty}\,\|v^a-v^b\|_{0,h},
\end{equation}
\begin{equation}\label{BasicHot5b}
|\gamma_{\delta}(v^a)-\gamma_{\delta}(v^b)|_{1,h}\leq
2\,\mol'_{\delta,\infty}\,|v^a-v^b|_{1,h}
+2\,\mol''_{\delta,\infty}
\,|\!|\!|\delta_hv^b|\!|\!|_{\infty,h}\,\|v^a-v^b\|_{0,h},
\end{equation}
\begin{equation}\label{BasicHot6}
\begin{split}
\|\zeta_{\delta}\|_{0,h}
\leq&\,2\,\mol''_{\delta,\infty}\,\,|z^a-z^b|_{\infty,h}\,\|v^b-z^b\|_{0,h}
+2\,\left(\mol'_{\delta,\infty}
+\mol''_{\delta,\infty}\,\,|z^a-z^b|_{\infty,h}\right)
\,\|v^a-v^b-z^a+z^b\|_{0,h}\\
\end{split}
\end{equation}
and
\begin{equation}\label{BasicHot7}
\begin{split}
|\zeta_{\delta}|_{1,h}
\leq&\,2\,(\mol'_{\delta,\infty}
+\mol''_{\delta,\infty}\,|z^a-z^b|_{\infty,h})
|v^a-v^b-z^a+z^b|_{1,h}\\
&\,+\mol''_{\delta,\infty}
\,(|v^a|_{1,h}+|v^b|_{1,h})\,|v^a-v^b-z^a+z^b|_{\infty,h}\\
&\,+2\,\mol''_{\delta,\infty}\,|z^a-z^b|_{\infty,h}\,|v^b-z^b|_{1,h}\\
&+2\,{\mathcal F}_{\delta}^{\star}(z^a,z^b)\,
\,\left(\,\|v^a-v^b-z^a+z^b\|_{0,h}+\|v^b-z^b\|_{0,h}\,\right),\\
\end{split}
\end{equation}
where $\zeta_{\delta}:=\gamma_{\delta}(v^{a})-\gamma_{\delta}(v^{b})
-\gamma_{\delta}(z^{a})+\gamma_{\delta}(z^{b})\in\gspace$ and
\begin{equation*}
\begin{split}
{\mathcal F}_{\delta}^{\star}(z^a,z^b):=&\,
\mol''_{\delta,\infty}\,|\!|\!|\delta_h(z^a-z^b)|\!|\!|_{\infty,h}
+\mol'''_{\delta,\infty}\,|z^a-z^b|_{\infty,h}
\,\left(|\!|\!|\delta_h(z^a-z^b)|\!|\!|_{\infty,h}
+|\!|\!|\delta_hz^b|\!|\!|_{\infty,h}\right).\\
\end{split}
\end{equation*}
\end{lem}
%
%
%
%
\begin{proof}
First, we observe that \eqref{BasicHot5} follows, easily, from
\eqref{MLF_5}. Now, let $\|\cdot\|_{\star}=\|\cdot\|_{0,h}$ or $|\cdot|_{1,h}$.
Then, using \eqref{MLF_2}, we have 
\begin{equation*}
\|\gamma_{\delta}(v^a)-\gamma_{\delta}(v^b)\|_{\star}\leq
\|\mol_{\delta}(\Ree(v^a))-\mol_{\delta}(\Ree(v^b))\|_{\star}
+\|\mol_{\delta}(\Imm(v^a))-\mol_{\delta}(\Imm(v^b))\|_{\star}
\end{equation*}
and
\begin{equation*}
\begin{split}
\|\zeta_{\delta}\|_{\star}\leq&\,\|\mol_{\delta}(\Ree(v^a))-\mol_{\delta}(\Ree(v^b))
-\mol_{\delta}(\Ree(z^a))+\mol_{\delta}(\Ree(z^b))\|_{\star}\\
&\,+\|\mol_{\delta}(\Imm(v^a))-\mol_{\delta}(\Imm(v^b))
-\mol_{\delta}(\Imm(z^a))+\mol_{\delta}(\Imm(z^b))\|_{\star},\\
\end{split}
\end{equation*}
which, along \eqref{BasicHot2}, \eqref{BasicHot3} and \eqref{BasicHot4}
(with ${\mathfrak g}=\mol_{\delta}$), 
yield \eqref{BasicHot5b}, \eqref{BasicHot6} and \eqref{BasicHot7}.
\end{proof}
%
%
\subsection{Space discrete operators}
Let $I_h:\dspace\mapsto\dspace$ be the identity operator and
$A_h,T_h:\dspace\mapsto\dspace$ be linear operators
defined by
$A_h:=I_h-\iii\,\tfrac{\tau}{2}\,\Delta_h$,
$T_h:=I_h+\iii\,\tfrac{\tau}{2}\,\Delta_h$
and $B_h:=A_h^{-1}T_h$.
%
%
%
\begin{lem}\label{operator_lemma}
The operators $A_h$ and $T_h$ are invertible and
the following relations hold
\begin{equation}\label{Megatree2}
\|A_h^{-1}(\chi)\|_{0,h}\leq\|\chi\|_{0,h},
\end{equation}
\begin{equation}\label{Megatree1}
\|B_h(\chi)\|_{0,h}=\|\chi\|_{0,h},
\end{equation}
\begin{equation}\label{Megatree5}
|B_h(\chi)|_{1,h}=|\chi|_{1,h}
\end{equation}
for $\chi\in\dspace$, and
%
%
\begin{equation}\label{Megatree4}
(I_h+B_h)^{-1}A_h^{-1}=\tfrac{1}{2}\,I_h.
\end{equation}
\end{lem}
%
%
%
%
\begin{proof}
Let $O_h=T_h$ or $A_h$.
In view of \eqref{NewEra2}, we obtain
\begin{equation}\label{proodos2}
\Ree[(O_h\chi,\chi)_{0,h}]=\|\chi\|_{0,h}^2\quad\forall\chi\in\dspace.
\end{equation}
Then, using \eqref{proodos2}, we, easily, conclude that ${\rm Ker}(O_h)=\{0\}$,
which, along with the finite dimensionality of $O_h$, yields that $O_h$ is
invertible.
Now, using \eqref{proodos2} and the Cauchy-Schwarz inequality, we
obtain
\begin{equation*}
\begin{split}
\|A_h^{-1}\chi\|_{0,h}^2=&\,\Ree[(A_hA_h^{-1}\chi,A_h^{-1}\chi)_{0,h}]\\
\leq&\,|(\chi,A_h^{-1}\chi)_{0,h}|\\
\leq&\|\chi\|_{0,h}\,\|A_h^{-1}\chi\|_{0,h}\quad\forall\chi\in\dspace,\\
\end{split}
\end{equation*}
which obviously yields \eqref{Megatree2}.
\par
Let $\chi\in\dspace$ and $v=B_h(\chi)$. Then, we have $A_h(v)=T_h(\chi)$ which is
equivalent to $v-\chi=\iii\,\frac{\tau}{2}\Delta_h(v+\chi)$. In view of \eqref{NewEra2},
we obtain $\Ree[(v-\chi,v+\chi)_{0,h}]=0$, or, equivalently 
$\|v\|_{0,h}^2=\|\chi\|_{0,h}^2$, and thus \eqref{Megatree1} is established. Since
$\Delta_h(v-\chi)=\iii\,\tfrac{\tau}{2}\Delta_h(\Delta_h(v+\chi))$, in view of 
\eqref{NewEra1} and \eqref{NewEra2}, we have
\begin{equation*}
\begin{split}
|v|_{1,h}^2-|\chi|_{1,h}^2=&\,\Ree[(\!\!(\delta_h(v-\chi),\delta_h(v+\chi))\!\!)_{0,h}]\\
=&\,-\Ree[(\Delta_h(v-\chi),v+\chi)_{0,h}]\\
=&\,-\Ree[\iii\,\tfrac{\tau}{2}\,(\Delta_h(\Delta_h(v+\chi)),v+\chi)_{0,h}]=0,\\
=&\,-\Ree[\iii\,\tfrac{\tau}{2}\,\|\Delta_h(v+\chi)\|_{0,h}]=0,\\
\end{split}
\end{equation*}
which, obviously, yields \eqref{Megatree5}.
\par
Finally, we obtain \eqref{Megatree4} proceeding as follows
%
%
%
\begin{equation*}
\begin{split}
(I_h+B_h)^{-1}A_h^{-1}=&\,[A_h^{-1}(A_h+T_h)]^{-1}A_h^{-1}\\
=&\,(A_h+T_h)^{-1}\,A_h\,A_h^{-1}\\
=&\,(A_h+T_h)^{-1}\\
=&\,(2\,I_h)^{-1}\\
=&\,\tfrac{1}{2}\,I_h.\\
\end{split}
\end{equation*}
%
%
%
%
%
%
%
\end{proof}
%
%
%
\section{Discretization Errors}\label{Section3}
%
\subsection{Consistency of the discretization in time}
To simplify the notation, we set $t^{\quarter}:=\tfrac{\tau}{4}$, 
$\uu^{\frac{1}{4}}:={\sf I}_h[\uu(t^{\quarter},\cdot)]$,
$\uu^{n}:={\sf I}_h[\uu(t_{n},\cdot)]$
for $n=0,\dots,N$, and $\uu^{n+\half}:={\sf I}_h[\uu(t^{n+\half},\cdot)]$
for $n=0,\dots,N-1$.
In view of the Dirichlet boundary conditions \eqref{PSL_b} and
the compatibility conditions \eqref{PSL_d},
it holds that $\uu^{\frac{1}{4}}\in\dspace$, $\uu^n\in\dspace$ for $n=0,\dots,N$
and $\uu^{n+\half}\in\dspace$ for $n=0,\dots,N-1$.
Also, we simplify the notation by setting $w^{\ssy A}=g(|u|^2)$ and
$w^{\ssy B}=g(|u|^2)\,u$.
\par
For $n=1,\dots,N-1$, we define ${\sf r}^n\in\dspace$ by
\begin{equation}\label{NORAD_22}
\tfrac{1}{2}\,\left[g\big(|\uu^{n+\half}|^2\big)+g\big(|\uu^{n-\half}|^2\big)\right]
=g\big(|\uu^n|^2\big)+{\sf r}^n.
\end{equation}
Then, applying the Taylor formula, in a standard way, we obtain
\begin{equation*}\label{USNORTH_21}
{\sf r}^n_j=\tfrac{\tau^2}{2}\,\int_{0}^{\half}
\left[\,(\tfrac{1}{2}-s)\,\partial_t^2w^{\ssy A}(t_n+s\,\tau,x_j)
+s\,\partial_t^2w^{\ssy A}(t^{n-\half}+s\,\tau,x_j)\right]\;ds,
\quad j=0,\dots,J+1,
\end{equation*}
for $n=1,\dots,N-1$, which, easily, yields
\begin{gather}
\max_{1\leq{n}\leq{\ssy N-1}}\|{\sf r}^n\|_{0,h}
\leq{\widehat{\sf C}}_{1,{\ssy A}}\,\tau^2
\,\,\max_{\ssy{[0,T]\times\III}}|\partial_t^2w^{\ssy A}|,\label{SHELL_21}\\
\max_{1\leq{n}\leq{\ssy N-1}}|{\sf r}^n|_{1,h}
\leq{\widehat{\sf C}}_{1,{\ssy B}}\,\tau^2
\,\,\max_{\ssy{[0,T]\times\III}}|\partial_x\partial_t^2w^{\ssy A}|,\label{SHELL_21H1}\\
\max_{2\leq{n}\leq{\ssy N-1}}
\|{\sf r}^n-{\sf r}^{n-1}\|_{0,h}
\leq{\widehat{\sf C}}_{2,{\ssy A}}\,\tau^3
\,\max_{\ssy{[0,T]\times\III}}|\partial_t^3w^{\ssy A}|,\label{SHELL_22}\\
\max_{2\leq{n}\leq{\ssy N-1}}
|{\sf r}^n-{\sf r}^{n-1}|_{1,h}
\leq{\widehat{\sf C}}_{2,{\ssy B}}\,\tau^3
\,\max_{\ssy{[0,T]\times\III}}|\partial_x\partial_t^3w^{\ssy A}|.\label{SHELL_22H1}
\end{gather}
\par
Let ${\sf r}^{\frac{1}{4}}\in\gspace$ be defined by
\begin{equation}\label{NORAD_00}
\uu^{\half}-\uu^0=\iii\,\tfrac{\tau}{2}\,{\sf I}_h\left[
\tfrac{\uu_{xx}(t^{\half},\cdot)+\uu_{xx}(t_0,\cdot)}{2}\right]
+\iii\,\tfrac{\tau}{2}\,g(|\uu^0|^2)\otimes\tfrac{\uu^{\half}+\uu^0}{2}
+\tfrac{\tau}{2}
\,{\sf I}_h\left[\tfrac{f(t^{\half},\cdot)+f(t_0,\cdot)}{2}\right]
+\tfrac{\tau}{2}\,{\sf r}^{\frac{1}{4}}
\end{equation}
and ${\sf r}^{n+\half}\in\gspace$ be given by
\begin{equation}\label{NORAD_02}
\uu^{n+1}-\uu^n=\iii\,\tau
\,{\sf I}_h\left[\tfrac{\uu_{xx}(t_{n+1},\cdot)+\uu_{xx}(t_n,\cdot)}{2}\right]
+\iii\,\tau\,g(|\uu^{n+\half}|^2)\otimes\tfrac{\uu^{n+1}+\uu^n}{2}
+\tau\,{\sf I}_h\left[\tfrac{f(t_{n+1},\cdot)+f(t_n,\cdot)}{2}\right]
+\tau\,{\sf r}^{n+\half}
\end{equation}
for $n=0,\dots,N-1$. Assuming that the solution $\uu$ is
smooth enough on $[0,T]\times{\mathcal I}$, and using
\eqref{PSL_d} and the Dirichlet boundary conditions \eqref{PSL_b}, we conclude
that $u_{xx}(t,x)=-f(t,x)$ for $t\in[0,T]$ and $x\in\{x_a,x_b\}$.
Thus, we have ${\sf r}^{\frac{1}{4}}\in\dspace$ and ${\sf r}^{n+\half}\in\dspace$ for $n=0,\dots,N-1$.
\par
Combining \eqref{PSL_a} with a standard application 
of the Taylor formula we get the following formulas:
\begin{equation}\label{USNORTH_02}
\begin{split}
{\sf r}^{n+\half}_j=&\,\tfrac{\tau^2}{2}\,\int_0^{\half}
\left[s^2\,\partial_t^3\uu(t_n+s\,\tau,x_j)
+(\tfrac{1}{2}-s)^2\,\partial_t^3\uu(t^{n+\half}+s\,\tau,x_j)\right]\;ds\\
&\,-\tfrac{\tau^2}{2}\,\int_0^{\frac{1}{2}}
\left[s\,\partial_t^3\uu(t_n+s\,\tau,x_j)
+(\tfrac{1}{2}-s)\,\partial_t^3\uu(t^{n+\half}+s\,\tau,x_j)\right]\;ds\\
&\,+\iii\,\tfrac{\tau^2}{2}\,\int_0^{\frac{1}{2}}
\left[s\,\partial_t^3w^{\ssy B}(t_n+s\,\tau,x_j)
+(\tfrac{1}{2}-s)\,\partial_t^3w^{\ssy B}(t^{n+\half}+s\,\tau,x_j)\right]\;ds\\
&\,-\iii\,\tfrac{g(|\uu(t^{n+\half},x_j)|^2)}{2}\,\tau^2
\,\int_0^{\frac{1}{2}}\left[s\,\uu_{tt}(t_n+s\,\tau,x_j)
+(\tfrac{1}{2}-s)\,\uu_{tt}(t^{n+\half}+s\,\tau,x_j)\right]\;ds\\
\end{split}
\end{equation}
for $j=0,\dots,J+1$ and $n=0,\dots,N-1$, and
\begin{equation}\label{USNORTH_04}
\begin{split}
{\sf r}^{\quarter}_j=&\,\tfrac{\tau^2}{2}
\,\int_0^{\quarter}\left[\,s^2\,\partial_t^3\uu(s\,\tau,x_j)
+(\tfrac{1}{4}-s)^2\,\partial_t^3\uu(t^{\quarter}+s\,\tau,x_j)\,\right]\;ds\\
&\,-\tfrac{\tau^2}{2}\,\int_0^{\frac{1}{4}}\left[s\,\partial_t^3\uu(s\,\tau,x_j)
+(\tfrac{1}{4}-s)\,\partial_t^3\uu(t^{\quarter}+s\,\tau,x_j)\right]\;ds\\
&\,-\iii\,g(|\uu_j^{\quarter}|^2)
\,\tfrac{\tau^2}{2}\,\int_0^{\frac{1}{4}}
\left[s\,\partial_t^2\uu(s\,\tau,x_j)
+(\tfrac{1}{4}-s)\,\partial_t^2\uu(t^{\quarter}+s\,\tau,x_j)\right]\;ds\\
&\,+\iii\,\tfrac{\tau^2}{2}
\,\int_0^{\frac{1}{4}}\left[s\,\partial_t^2w^{\ssy B}(s\,\tau,x_j)
+(\tfrac{1}{4}-s)\,\partial_t^2w^{\ssy B}(t^{\quarter}+s\,\tau,x_j)\right]\;ds\\
&\,+\iii\,\tfrac{\uu_j^{\frac{1}{2}}+\uu^0_j}{2}\,\tau\,
\int_{0}^{\quarter}\partial_tw^{\ssy A}(s\,\tau,x_j)\;ds\\
\end{split}
\end{equation}
for $j=0,\dots,J+1$.
From \eqref{USNORTH_02} and \eqref{USNORTH_04}, we obtain:
\begin{gather}
\|{\sf r}^{\quarter}\|_{0,h}\leq\,{\widehat{\sf C}}_{3,{\ssy A}}\,\tau\,
\max_{\ssy[0,T]\times I}(|\partial_t^2u|+|\partial_t^3u|+|\partial_tw^{\ssy A}|
+|\partial_t^2w^{\ssy B}|),\label{SHELL_01}\\
|{\sf r}^{\quarter}|_{1,h}\leq\,{\widehat{\sf C}}_{3,{\ssy B}}\,\tau\,
\max_{\ssy[0,T]\times I}(|\partial_x\partial_t^2u|
+|\partial_x\partial_t^3u|
+|\partial_x\partial_tw^{\ssy A}|
+|\partial_x\partial_t^2w^{\ssy B}|),\label{SHELL_01H1}\\
\max_{0\leq{n}\leq{\ssy N-1}}
\|{\sf r}^{n+\half}\|_{0,h}\leq\,{\widehat{\sf C}}_{4,{\ssy A}}\,\tau^2\,
\max_{\ssy[0,T]\times I}(|\partial_t^2u|+|\partial_t^3u|
+|\partial_t^3w^{\ssy B}|),\label{SHELL_02}\\
\max_{0\leq{n}\leq{\ssy N-1}}
|{\sf r}^{n+\half}|_{1,h}\leq\,{\widehat{\sf C}}_{4,{\ssy B}}\,\tau^2\,
\max_{\ssy[0,T]\times I}(|\partial_x\partial_t^2u|+|\partial_x\partial_t^3u|
+|\partial_x\partial_t^3w^{\ssy B}|),\label{SHELL_02H1}\\
\max_{1\leq{n}\leq{\ssy N-1}}
\|{\sf r}^{n+\half}-{\sf r}^{(n-1)+\half}\|_{0,h}\leq\,{\widehat{\sf C}}_{5,{\ssy A}}\,\tau^3
\,\max_{\ssy[0,T]\times I}(|\partial_t^3u|+|\partial_t^4u|+|\partial_t^4w^{\ssy B}|),
\label{SHELL_03}\\
\max_{1\leq{n}\leq{\ssy N-1}}
|{\sf r}^{n+\half}-{\sf r}^{(n-1)+\half}|_{1,h}\leq\,{\widehat{\sf C}}_{5,{\ssy B}}\,\tau^3
\,\max_{\ssy[0,T]\times I}(|\partial_x\partial_t^3u|
+|\partial_x\partial_t^4u|+|\partial_x\partial_t^4w^{\ssy B}|).
\label{SHELL_03H1}
\end{gather}
%
%
%
%
%
%
%
%
\subsection{Approximation estimates for the Discrete Elliptic Projection}
Let $v\in C^4(I;\cset)$. After applying the Taylor formula around $x=x_j$, it
follows that 
\begin{equation}\label{ELP_2}
\Delta_h(\PP v)-\PP(v'')=\tfrac{h^2}{12}\,{\sf r}^{\ssy{\rm E}}(v),
\end{equation}
where ${\sf r}^{\ssy{\rm E}}(v)\in\dspace$ is defined by
\begin{equation}\label{ELP_3}
({\sf r}^{\ssy{\rm E}}(v))_j:=\int_0^1\left[\,(1-y)^3\,v''''(x_j+h\,y)
+y^3\,v''''(x_{j-1}+h\,y)\,\right]\;dy,
\quad j=1,\dots,J.
\end{equation}
Subtracting \eqref{ELP_1} from \eqref{ELP_2}, we get the following
error equation
\begin{equation}\label{ELP_4}
\Delta_h(\PP v-{\sf R}_h[v])=\tfrac{h^2}{12}\,{\sf r}^{\ssy{\rm E}}(v).
\end{equation}
Taking the inner product of both sides of \eqref{ELP_4}
with $(\PP(v)-{\sf R}_h[v])$  and using \eqref{NewEra2},
the Cauchy-Schwarz inequality and \eqref{dpoincare}
we obtain
\begin{equation}\label{ELP_5_0}
|{\sf R}_h[v]-\PP(v)|_{1,h}\leq\,\tfrac{{\sf L}}{12}
\,h^2\,\|{\sf r}^{\ssy{\rm E}}(v)\|_{0,h}\\
\end{equation}
which, along with \eqref{dpoincare} and \eqref{ELP_3}, yields
\begin{equation}\label{ELP_5}
\|{\sf R}_h[v]-\PP(v)\|_{1,h}
\leq\,\tfrac{\sqrt{1+{\sf L}^2}\,{\sf L}^{\frac{3}{2}}}{24}
\,h^2\,\max_{\ssy I}|v''''|.
\end{equation}
\par   
We close the section with a useful lemma.
%
%
\begin{lem}
Let $w\in C_{t,x}^{1,0}(\QQ)$ and
$\partial_tw\in C^{0,4}_{t,x}(\QQ)$.
Then, it holds that
\begin{equation}\label{ELP_6}
\Big\|{\sf R}_h\left[\tfrac{w(t,\cdot)-w(s,\cdot)}{t-s}\right]
-\PP\left[\tfrac{w(t,\cdot)-w(s,\cdot)}{t-s}\right]\Big\|_{1,h}
\leq\,\tfrac{\sqrt{1+{\sf L}^2}\,{\sf L}^{\frac{3}{2}}}{24}
\,h^2\,\max_{\ssy\QQ}|\partial_x^4\partial_tw|
\end{equation}
for all $t,s\in[0,T]$ with $s<t$.
\end{lem}
%
%
%
\begin{proof}
For $t,s\in[0,T]$ with $s<t$, and $\psi_h:=
{\sf R}_h\left[\tfrac{w(t,\cdot)-w(s,\cdot)}{t-s}\right]
-\PP\left[\tfrac{w(t,\cdot)-w(s,\cdot)}{t-s}\right]$. Using
\eqref{dpoincare} and \eqref{ELP_5_0}, we have
\begin{equation*}
\begin{split}
\|\psi_h\|_{1,h}\leq&\,\sqrt{1+{\sf L}^2}\,|\psi_h|_{1,h}\\
\leq&\,\tfrac{\sqrt{1+{\sf L}^2}\,\sf L}{12}\,h^2\,
\Big\|{\sf r}^{\ssy{\rm E}}\left[\tfrac{w(t,\cdot)
-w(s,\cdot)}{t-s}\right]\Big\|_{0,h}\\
\leq&\,\tfrac{\sqrt{1+{\sf L}^2}\,\sf L}{12}\,h^2\,(t-s)^{-1}\,
\int_{s}^{t}\|{\sf r}^{\ssy{\rm E}}[\partial_tw({\hat s},\cdot)]\|_{0,h}\;d{\hat s}\\
\leq&\,\tfrac{\sqrt{1+{\sf L}^2}\,{\sf L}^{\frac{3}{2}}}{24}
\,h^2\,\max_{\ssy\QQ}|\partial_x^4\partial_tw|.\\
\end{split}
\end{equation*}
\end{proof}
\section{Convergence}\label{Section4}
%
%
\subsection{Existence and uniqueness of the (RFD)
approximations}
%
%
The following lemma establishes that the (RFD)
approximations are unconditionally well-defined.
\begin{lem}\label{Plan_1}
The (RFD) approximations $W^{\half}$
and $(W^{n})_{n=0}^{\ssy N}$
defined by \eqref{BRS_1}, \eqref{BRS_12}, \eqref{BRS_2}
and \eqref{BRS_4} are well-defined.
\end{lem}
\begin{proof}
Let $\phi\in\rspace$, $\varepsilon>0$ and
$\nu_h[\varepsilon,\phi]:\dspace\mapsto\dspace$ be a discrete
operator given by
\begin{equation*}
\nu_h[\varepsilon,\phi]\chi:=\chi-\iii\,\varepsilon\,\tau\,\Delta_h(\chi)
-\iii\,\varepsilon\,\tau\,\phi\otimes\chi
\quad\forall\chi\in\dspace.
\end{equation*}
Then, using \eqref{NewEra2}, we obtain that
$\Ree(\nu_h[\varepsilon,\phi]\chi,\chi)_{0,h}=\|\chi\|^2_{0,h}$
for $\chi\in\dspace$, which, easily, yields that
${\rm Ker}(\nu_h[\varepsilon,\phi])=\{0\}$. Since,
the space $\dspace$ is finite dimensional, we conclude that
$\nu_n[\varepsilon,\phi]$ is invertible.
\par
By \eqref{BRS_1} the initial approximation $W^0$ is clearly defined. 
According to \eqref{BRS_12}, \eqref{BRS_2} and \eqref{BRS_4}
we have
\begin{equation*}
W^{\half}=\nu_h^{-1}[\tfrac{1}{4},g(|u^0|^2)]\left[\,
W^0+\iii\,\tfrac{\tau}{4}\,\Delta_h(W^0)
+\iii\,\tfrac{\tau}{4}\,g(|u^0|^2)\otimes W^0
+\tfrac{\tau}{4}
\,\PP\left[f(t^{\half},\cdot)+f(t_0,\cdot)\right]
\,\right]
\end{equation*}
and
\begin{equation*}
W^{n+1}=\nu_h^{-1}[\tfrac{1}{2},\Phi^{n+\half}]\left[\,
W^n+\iii\,\tfrac{\tau}{2}\,\Delta_h(W^n)
+\iii\,\tfrac{\tau}{2}\,\Phi^{n+\half}\otimes W^n
+\tfrac{\tau}{2}
\,\PP\left[f(t_{n+1},\cdot)+f(t_n,\cdot)\right]
\,\right]
\end{equation*}
for $n=1,\dots,N-1$.
\par

\end{proof}
%
%
%
\subsection{The (MRFD) scheme}
%
For given $\delta>0$, the modified version of the (RFD) method 
derives $\delta-$dependent approximations of the solution
$\uu$ to \eqref{PSL_a}-\eqref{PSL_d} according to the steps below:
\par\vskip0.2truecm\par
{\tt Step I}: Define $V_{\delta}^0\in\dspace$  by
\begin{equation}\label{modCB1}
V_{\delta}^0:={\sf R}_h[\uu_0]
\end{equation}
and then find $V^{\half}_{\delta}\in\dspace$ such that
\begin{equation}\label{modCB1a}
\begin{split}
V^{\half}_{\delta}-V^0_{\delta}
=\iii\,\tfrac{\tau}{2}\,
\Delta_{h}\left(\,\tfrac{V^{\half}_{\delta}+V^0_{\delta}}{2}\,\right)
+\iii\,\tfrac{\tau}{2}
\,g\big(|\uu^0|^2\big)\otimes\tfrac{V^{\half}_{\delta}+V^0_{\delta}}{2}
+\tfrac{\tau}{2}
\,\PP\left[\tfrac{f(t^{\half},\cdot)+f(t_0,\cdot)}{2}\right].
\end{split}
\end{equation}
\par\vskip0.2truecm\par
{\tt Step I\!I}: Define $\Phi_{\delta}^{\half}\in\rspace$ by
\begin{equation}\label{modCB1b}
\Phi^{\half}_{\delta}:=g\big(\big|\gamma_{\delta}(V_{\delta}^{\half})\big|^2\big)
\end{equation}
and then find $V_{\delta}^1\in\dspace$ such that
\begin{equation}\label{modCB2}
\begin{split}
V^1_{\delta}-V^0_{\delta}= &\,
\iii\,\tau\,\Delta_h\left(\tfrac{V^1_{\delta}+V^0_{\delta}}{2}\right)
+\iii\,\tau\,\gff_{\delta}\big(\Phi^{\half}_{\delta}\big)
\otimes\gamma_{\delta}\left(\tfrac{V^1_{\delta}+V^0_{\delta}}{2}\right)
+\tau\,\PP\left[\tfrac{f(t_{1},\cdot)+f(t_0,\cdot)}{2}\right].\\
\end{split}
\end{equation}
\par\vskip0.2truecm\par
{\tt Step I\!I\!I}: For $n=1,\dots,N-1$,  define
$\Phi_{\delta}^{n+\half}\in\rspace$ by
\begin{equation}\label{modCB3a}
\Phi_{\delta}^{n+\half}:=2\,g\big(\big|\gamma_{\delta}(V_{\delta}^n)\big|^2\big)
-\Phi_{\delta}^{n-\half}
\end{equation}
and then find $V^{n+1}_{\delta}\in\dspace$ such that
\begin{equation}\label{modCB3b}
V^{n+1}_{\delta}-V^n_{\delta}=
\iii\,\tau\,\Delta_h\left(\tfrac{V^{n+1}_{\delta}+V^n_{\delta}}{2}\right)
+\iii\,\tau\,\gff_{\delta}\big(\Phi_{\delta}^{n+\half}\big)
\otimes\gamma_{\delta}\left(\tfrac{V^{n+1}_{\delta}+V^n_{\delta}}{2}\right)
+\tau\,\PP\left[\tfrac{f(t_{n+1},\cdot)+f(t_n,\cdot)}{2}\right].
\end{equation}
\begin{remark}
Under the light of Lemma~\ref{Plan_1} and
in view of \eqref{BRS_1}, \eqref{modCB1}, \eqref{BRS_12} and
\eqref{modCB1a} it follows that $V_{\delta}^0=\UUU^0$ and
$V_{\delta}^{\half}=\UUU^{\half}$.
\end{remark}

%
\subsection{Existence of the (MRFD) approximations}
First, we recall the following Brouwer-type fixed-point lemma,
for a proof of which we refer the reader to \cite{ADKar}. 
\begin{lem}\label{Plan_0}
Let $(\cset,H,(\cdot,\cdot)_{\ssy H})$ be a complex
finite dimensional inner product space,
$\|\cdot\|_{\ssy H}$ be the associated norm,
$\mu:H\mapsto H$ be a continuous operator,
${\mathcal S}_{\varepsilon}:=\{z\in H:\|z\|_{\ssy H}=\varepsilon\}$
for $\varepsilon>0$, and 
${\mathcal B}_{\varepsilon}:=\{z\in H:\|z\|_{\ssy H}\leq\varepsilon\}$
for $\varepsilon>0$.
If there exists $\alpha>0$ such that
$\Ree(\mu(z),z)_{\ssy H}\ge 0\quad\forall\,z\in{\mathcal S}_{\alpha}$,
then there exists $z_{\star}\in{\mathcal B}_{\alpha}$
such that $\mu(z_{\star})=0$.
\end{lem}
%
%

%
\begin{proposition}\label{Plan_2}
For $\delta>0$, there exist $(V_{\delta}^n)_{n=1}^{\ssy N}\subset\dspace$
satisfying \eqref{modCB2} and \eqref{modCB3b}.
\end{proposition}
%
%
%
%
\begin{proof}
Let $\varphi\in\rspace$, $z\in\dspace$
and $\mu_h:\dspace\mapsto\dspace$ be a continuous
non linear operator defined by
\begin{equation*}
\mu_h(\chi):=2\,\chi-\,\iii\,\tau\,\Delta_h(\chi)
-\iii\,\tau\,\left[\,\gff_{\delta}(\varphi)\otimes \gamma_{\delta}(\chi)\,\right]
+z\quad\forall\,\chi\in\dspace.
\end{equation*}
Let $\alpha>0$ and $\chi\in\dspace$ with $\|\chi\|_{0,h}=\alpha$.
Using \eqref{NewEra2}, the Cauchy-Schwarz inequality and \eqref{MLF_4},
we obtain
\begin{equation}\label{KAFE_01}
\begin{split}
\Ree(\mu_h(\chi),\chi)_{0,h}=&\,2\,\|\chi\|_{0,h}^2
+\tau\,\Imm[\left(\gff_{\delta}(\varphi)\otimes\gamma_{\delta}(\chi),
\chi\right)_{0,h}]+\Ree[(z,\chi)_{0,h}]\\
\ge&\,\|\chi\|_{0,h}\,\left(\,2\,\|\chi\|_{0,h}
-\tau\,|\gff_{\delta}(\varphi)|_{\infty,h}\,\|\gamma_{\delta}(\chi)\|_{0,h}
-\|z\|_{0,h}\,\right)\\
\ge&\,\alpha\,\left(\,2\,\alpha
-\tau\,\sqrt{2\,{\sf L}}\,\sup_{\ssy\rset}|\gff_{\delta}|^2
-\|z\|_{0,h}\,\right).\\
\end{split}
\end{equation}
Choosing $\alpha=\alpha_{\star}:=\tfrac{\tau\,\sqrt{2\,{\sf L}}}{2}
\,\sup_{\ssy\rset}|\gff_{\delta}|^2+\tfrac{1}{2}\,\|z\|_{0,h}+1$,
\eqref{KAFE_01} yields that 
$\Ree(\mu_h(\chi),\chi)_{0,h}>0$, which, in view of Lemma~\ref{Plan_0}, results that
there exists $\chi_{\star}\in\dspace$ such that
$\|\chi_{\star}\|_{0,h}\leq\alpha_{\star}$ and $\mu_h(\chi_{\star})=0$.
%
%
\par  
We establish the existence of the modified approximations by induction.
First, we observe that $V_{\delta}^0$ is well defined. Next, we assume
that there exists a modified approximation $V_{\delta}^n$ for a
$n\in\{0,\dots,N-1\}$. Then, we choose $\varphi=\Phi^{n+1}_{\delta}$ and
$z=2\,V_{\delta}^n+\tfrac{\tau}{2}\,\PP\left[f(t_{n+1},\cdot)+f(t_n,\cdot)\right]$,
to obtain a root $\chi^{n+1}_{\star}\in\dspace$ of the corresponding operator $\mu_h$.
Thus, $V^{n+1}_{\delta}=2\,\chi^{n+1}_{\star}-V^n_{\delta}$ forms a solution to the
nonlinear system \eqref{modCB3b} when $n\ge1$, or, \eqref{modCB2} when $n=0$.
\end{proof}
%
%
%
\subsection{Convergence of the (MRFD) scheme}\label{modCB_Conv}
%
In this section we investigate the convergence properties
of the modified (RFD) approximations.
%
%
%
\begin{thm}\label{modCB_Conv}
Let $u_{\ssy\max}:=\max_{\ssy\QQ}|\uu|$,
$g_{\ssy\max}:=\max_{\ssy\QQ}|g(|\uu|^2)|$
and $\ddelta\ge\max\{u_{\ssy\max},g_{\ssy\max}\}$.
Then, there exist positive constants ${\sf C}^{\ssy{\sf A}}_{\ssy\ddelta}$,
${\sf C}^{\ssy{\sf B}}_{\ssy\ddelta}$ and ${\sf C}_{\ssy\ddelta}^{\ssy{\sf C}}$,
independent of $\tau$ and $h$, such that: if
$\tau\,{\sf C}^{\ssy{\sf A}}_{\ssy\ddelta}\leq\tfrac{1}{2}$, then
\begin{equation}\label{mod_BR_cnv_1}
\|u^{\half}-V_{\ddelta}^{\half}\|_{1,h}
+\max_{0\leq{m}\leq{\ssy N}}\|\uu^m-V_{\ddelta}^m\|_{1,h}
\leq\,{\sf C}_{\ssy\ddelta}^{\ssy{\sf B}}\,(\tau^2+h^2)
\end{equation}
and
\begin{equation}\label{mod_BR_cnv_2}
\max_{0\leq{m}\leq{\ssy N-1}}\|g(\uu^{m+\half})-\Phi_{\ddelta}^{m+\half}\|_{1,h}
\leq\,{\sf C}_{\ssy\ddelta}^{\ssy{\sf C}}\,(\tau^2+h^2).
\end{equation}
\end{thm}
%
%
%
\begin{proof}
To simplify the notation, we set
$\emid^{m}:=g(|\uu^{m+\half}|^2)-\Phi_{\ddelta}^{m+\half}\in\rspacez$
for $m=0,\dots,N-1$,
$e^{\half}:=u^{\half}-V_{\ddelta}^{\half}\in\dspace$,
$\varrho^{\half}:=u^{\half}-{\sf R}_h[u(t^{\half},\cdot)]\in\dspace$,
$\vq^{\half}:={\sf R}_h[u(t^{\half},\cdot)]-V_{\ddelta}^{\half}\in\dspace$,
$\varrho^m:=\uu^m-{\sf R}_h[u(t_m,\cdot)]\in\dspace$,
$\vq^m:={\sf R}_h[u(t_m,\cdot)]-V_{\ssy\ddelta}^m\in\dspace$ and
$e^m:=\uu^m-V_{\ddelta}^m\in\dspace$ for $m=0,\dots,N$,
and $\partial\vq^{m}:=\frac{\vq^{m}-\vq^{m-1}}{\tau}\in\dspace$ for $m=1,\dots,N$.
Also, we note that  $e^{\half}=\varrho^{\half}+\vq^{\half}$,
$e^m=\varrho^m+\vq^m$ for $m=0,\dots,N$, and that due
to \eqref{modCB1} we have $\vq^0=0$.
\par
In the sequel,  we will use the symbol $C$ to denote a generic constant that is
independent of $\tau$, $h$ and $\ddelta$,  and may changes value from one line
to the other. Also, we will use the symbol $C_{\ddelta}$ (with or without additional
symbols) to denote a generic constant that depends on $\ddelta$ but is independent
of $\tau$, $h$, and may changes value from one line to the other. We note that
the constants $C$ and $C_{\ddelta}$ may depend on the solution $u$ and its
derivatives.
%
%
\par\noindent\vskip0.2truecm\par\noindent
$\boxed{{\tt Part\,\,\, 1}}:$ 
Combining \eqref{modCB1}, \eqref{modCB1a}, \eqref{NORAD_00}
and \eqref{ELP_1}, we get the following error equation:
\begin{equation}\label{LOX_NES_1}
\vq^{\half}-\vq^0-\iii\,\tfrac{\tau}{4}\,\Delta_h(\vq^{\half}+\vq^0)
=\tau\,\sum_{\ell=1}^{3}{\mathcal A}^{\ell},
\end{equation}
where
\begin{equation*}
\begin{split}
{\mathcal A}^{1}:=&\,\tfrac{1}{2}
\,\left(\,{\sf R}_h\left[\tfrac{\uu(t^{\half},\cdot)-\uu(t_0,\cdot)}{(\tau/2)}\right]
-\tfrac{\uu^{\half}-\uu^0}{(\tau/2)}\,\right)\in\dspace,\\
{\mathcal A}^{2}:=&\,\tfrac{1}{2}\,{\sf r}^{\frac{1}{4}}\in\dspace,\\
{\mathcal A}^{3}:=&\,\tfrac{\iii}{4}\,g\big(|u^0|^2\big)
\otimes(e^{\half}+e^0)\in\dspace.\\
\end{split}
\end{equation*}
%
%
\par\noindent\vskip0.2truecm\par\noindent
$\boxed{{\tt Part\,\,\, 2}}:$ 
Since $\vq^0=0$, after taking the $(\cdot,\cdot)_{0,h}-$inner product
of \eqref{LOX_NES_1} with $\vq^{\half}$, and then using
\eqref{NewEra2} and keeping the real parts of the relation obtained,
we arrive at
\begin{equation}\label{Les_Halles_0}
\|\vq^{\half}\|_{0,h}^2=\tau\,
\sum_{\ell=1}^3\Ree[({\mathcal A}^{\ell},\vq^{\half})_{0,h}].
\end{equation}
Using \eqref{ELP_6} and \eqref{SHELL_01}, we obtain
\begin{equation}\label{Les_Halles_1a}
\begin{split}
\|{\mathcal A}^1\|_{0,h}\leq&\,\tfrac{1}{2}\,
\left\|\,{\sf R}_h\left[\tfrac{\uu(t^{\half},\cdot)-\uu(t_0,\cdot)}{(\tau/2)}\right]
-\tfrac{\uu^{\half}-\uu^0}{(\tau/2)}\,\right\|_{0,h}\\
\leq&\,C\,h^2\\
\end{split}
\end{equation}
and
\begin{equation}\label{Les_Halles_1b}
\begin{split}
\|{\mathcal A}^{2}\|_{0,h}
\leq&\,\tfrac{1}{2}\,\|{\sf r}^{\frac{1}{4}}\|_{0,h}\\
\leq&\,C\,\tau.\\
\end{split}
\end{equation}
Combining the Cauchy-Schwarz inequality with
\eqref{Les_Halles_1a} and \eqref{Les_Halles_1b}, we get
\begin{equation}\label{Les_Halles_1c}
\Ree[({\mathcal A}^{1},\vq^{\half})_{0,h}]
+\Ree[({\mathcal A}^{2},\vq^{\half})_{0,h}]
\leq\,C\,(\tau+h^2)\,\|\vq^{\frac{1}{2}}\|_{0,h}.
\end{equation}
Observing that
\begin{equation*}
\begin{split}
\Ree[({\mathcal A}^{3},\vq^{\half})_{0,h}]=&\,-\tfrac{1}{4}
\,\Imm[(g(|u^0|^2)\otimes(\varrho^{\half}+\vq^{\half}
+\varrho^0+\vq^0),\vq^{\half})_{0,h}]\\
=&\,-\tfrac{1}{4}
\,\Imm[(g(|u^0|^2)\otimes(\varrho^{\half}+\varrho^0),\vq^{\half})_{0,h}]
-\tfrac{1}{4}
\,\Imm[(g(|u^0|^2)\otimes\vq^{\half},\vq^{\half})_{0,h}]\\
=&\,-\tfrac{1}{4}
\,\Imm[(g(|u^0|^2)\otimes(\varrho^{\half}+\varrho^0),\vq^{\half})_{0,h}],\\
\end{split}
\end{equation*}
we use the Cauchy-Schwarz inequality and \eqref{ELP_5}, to get
\begin{equation}\label{Les_Halles_2}
\begin{split}
\Ree[({\mathcal A}^{3},\vq^{\half})_{0,h}]
\leq&\,C\,(\,\|\varrho^{\half}\|_{0,h}
+\|\varrho^0\|_{0,h}\,)
\,\|\vq^{\half}\|_{0,h}\\
\leq&\,C\,h^2\,\|\vq^{\half}\|_{0,h}.\\
\end{split}
\end{equation}
In view of \eqref{Les_Halles_0}, \eqref{Les_Halles_1c}
and \eqref{Les_Halles_2}, we, easily, conclude that
\begin{equation}\label{Les_Halles_3}
\|\vq^{\half}\|_{0,h}
\leq\,C\,(\tau^2+\tau\,h^2).
\end{equation}
\par
Taking the $(\cdot,\cdot)_{0,h}-$inner product
of \eqref{LOX_NES_1} with $\Delta_h(\vq^{\half})$, and then using
\eqref{NewEra1} and keeping the real parts of the relation
obtained, it follows that
\begin{equation}\label{Les_Halles_10}
|\vq^{\half}|_{1,h}^2=\tau\,
\sum_{\ell=1}^3\Ree[(\!\!(\delta_h{\mathcal A}^{\ell},\delta_h\vq^{\half})\!\!)_{0,h}].
\end{equation}
Using the Cauchy-Schwarz inequality, \eqref{ELP_6}, \eqref{SHELL_01H1}
and \eqref{ELP_5}, we have
\begin{equation}\label{Christmass_19a}
\begin{split}
\Ree[(\!\!(\delta_h{\mathcal A}^{1},\delta_h\vq^{\half})\!\!)_{0,h}]
+\Ree[(\!\!(\delta_h{\mathcal A}^{2},\delta_h\vq^{\half})\!\!)_{0,h}]
\leq&\,(|{\mathcal A}^1|_{1,h}
+|{\mathcal A}^2|_{1,h})\,|\vq^{\frac{1}{2}}|_{1,h}\\
\leq&\,C\,(h^2+\tau)\,|\vq^{\frac{1}{2}}|_{1,h}\\
\end{split}
\end{equation}
and
\begin{equation}\label{Christmass_19b}
\begin{split}
\Ree[(\!\!(\delta_h({\mathcal A}^3),\delta_h\vq^{\half})\!\!)_{0,h}]
=&\,-\tfrac{1}{4}
\,\Imm[(\!\!(\delta_h(g(|u^0|^2)\otimes(\varrho^{\half}+\varrho^0),
\delta_h\vq^{\half})\!\!)_{0,h}]\\
\leq&\,|g(|u^0|^2)\otimes(\varrho^{\half}+\varrho^0)|_{1,h}
\,|\vq^{\frac{1}{2}}|_{1,h}\\
\leq&\,\left[\,|g(|u^0|^2)|_{\infty,h}\,|\varrho^{\half}+\varrho^0|_{1,h}
+|\!|\!|\delta_h(g(|u^0|^2))|\!|\!|_{\infty,h}\,\|\varrho^{\half}+\varrho^0\|_{0,h}\,\right]
\,|\vq^{\frac{1}{2}}|_{1,h}\\
\leq&\,C\,(\|\varrho^{\half}\|_{1,h}+\|\varrho^0\|_{1,h})\,|\vq^{\frac{1}{2}}|_{1,h}\\
\leq&\,C\,h^2\,|\vq^{\frac{1}{2}}|_{1,h}.\\
\end{split}
\end{equation}
From \eqref{Les_Halles_10}, \eqref{Christmass_19a}
and \eqref{Christmass_19b}, it follows that
\begin{equation}\label{Les_Halles_11}
|\vq^{\half}|_{1,h}\leq\,C\,(\tau^2+\tau\,h^2).
\end{equation}
\par
Since $\ddelta\ge u_{\ssy\max}$, we use \eqref{modCB1b}, \eqref{MLF_3},
\eqref{MLF_4}, \eqref{BasicHot1} (with ${\mathfrak g}=g$ and
$\varepsilon=2\,\sup_{\ssy\rset}|\mol_{\ddelta}|^2$),
\eqref{Jalapeno1} and \eqref{BasicHot5}, to have
\begin{equation*}
\begin{split}
\|\emid^0\|_{0,h}=&\,\big\|g(|\gamma_{\ddelta}(u^{\half})|^2)
-g(|\gamma_{\ddelta}(V_{\ddelta}^{\half})|^2)\big\|_{0,h}\\
\leq&\,C_{\ddelta}\,
\,\big\|\,|\gamma_{\ddelta}(u^{\half})|^2
-|\gamma_{\ddelta}(V_{\ddelta}^{\half})|^2\,\big\|_{0,h}\\
\leq&\,C_{\ddelta}\,\left[|\gamma_{\ddelta}(u^{\half})|_{\infty,h}
+|\gamma_{\ddelta}(V_{\ddelta}^{\half})|_{\infty,h}\right]
\,\|\gamma_{\ddelta}(u^{\half})-\gamma_{\ddelta}(V_{\ddelta}^{\half})\|_{0,h}\\
\leq&\,C_{\ddelta}
\,\|\gamma_{\ddelta}(u^{\half})-\gamma_{\ddelta}(V_{\ddelta}^{\half})\|_{0,h}\\
\leq&\,C_{\ddelta}\,\|e^{\half}\|_{0,h}\\
\leq&\,C_{\ddelta}\,(\|\varrho^{\half}\|_{0,h}+\|\vq^{\half}\|_{0,h}),\\
\end{split}
\end{equation*}
which, along with \eqref{Les_Halles_3} and \eqref{ELP_5}, yields
\begin{equation}\label{Les_Halles_4}
\|\emid^0\|_{0,h}\leq\,C_{\ddelta}\,(\tau^2+h^2).
\end{equation}
Also, we use \eqref{MLF_4}, \eqref{BasicHot2} (with ${\mathfrak g}=g$ and
$\varepsilon=2\,\sup_{\ssy\rset}|\mol_{\ddelta}|^2$), \eqref{MLF_3},
\eqref{dpoincare}, \eqref{Jalapeno12}, \eqref{BasicHot5b}
and \eqref{LH1}, to get
\begin{equation*}
\begin{split}
|\emid^0|_{1,h}=&\,\big|g(|\gamma_{\ddelta}(V_{\ddelta}^{\half})|^2)
-g(|\gamma_{\ddelta}(u^{\half})|^2)\big|_{1,h}\\
\leq&\,C_{\ddelta}\,\left[\,
\big|\,|\gamma_{\ddelta}(V_{\ddelta}^{\half})|^2
-|\gamma_{\ddelta}(u^{\half})|^2\big|_{1,h}
+|\!|\!|\delta_h(|u^{\half}|^2)|\!|\!|_{\infty,h}
\,\||\gamma_{\ddelta}(V_{\ddelta}^{\half})|^2
-|\gamma_{\ddelta}(u^{\half})|^2\big\|_{0,h}\,\right]\\
\leq&\,C_{\ddelta}\,
\,\big||\gamma_{\ddelta}(V_{\ddelta}^{\half})|^2
-|\gamma_{\ddelta}(u^{\half})|^2\big|_{1,h}\\
\leq&\,C_{\ddelta}\,[\,1+|\gamma_{\ddelta}(V_{\ddelta}^{\frac{1}{2}})|_{\infty,h}
+|\!|\!|\delta_h(u^{\frac{1}{2}})|\!|\!|_{\infty,h}\,]
\,|\gamma_{\ddelta}(V_{\ddelta}^{\frac{1}{2}})
-\gamma_{\ddelta}(u^{\frac{1}{2}})|_{1,h}\\
\leq&\,C_{\ddelta}\,|\gamma_{\ddelta}(V_{\ddelta}^{\frac{1}{2}})
-\gamma_{\ddelta}(u^{\frac{1}{2}})|_{1,h}\\
\leq&\,C_{\ddelta}\,(1+|\!|\!|\delta_h(u^{\frac{1}{2}})|\!|\!|_{\infty,h})
\,|e^{\half}|_{1,h}\\
\leq&\,C_{\ddelta}\,
(|\varrho^{\half}|_{1,h}+|\vq^{\half}|_{1,h}),\\
\end{split}
\end{equation*}
which, along with \eqref{Les_Halles_11}
and \eqref{ELP_5}, yields
\begin{equation}\label{Les_Halles_4H1}
|\emid^0|_{1,h}\leq\,C_{\ddelta}\,(\tau^2+h^2).
\end{equation}
%
%
%
%
%
$\boxed{{\tt Part\,\,\,3}:}$
Since $\ddelta\ge\max\{u_{\ssy\max},g_{\ssy\max}\}$,
using \eqref{modCB2}, \eqref{modCB3b}, \eqref{NORAD_02}
and \eqref{ELP_1}, we get
\begin{equation}\label{ABBA_0}
\vq^{n+1}-\vq^n-\iii\,\tfrac{\tau}{2}\,\Delta_h(\vq^{n+1}+\vq^n)
=\tau\,\sum_{\ell=1}^{4}{\mathcal B}^{\ell,n},\quad n=0,\dots,N-1,
\end{equation}
where
\begin{equation}\label{Beta_Defs}
\begin{split}
{\mathcal B}^{1,n}:=&\,\left(\,
{\sf R}_h\left[\tfrac{\uu(t_{n+1},\cdot)-\uu(t_n,\cdot)}{\tau}\right]
-\tfrac{\uu^{n+1}-\uu^n}{\tau}\,\right)\in\dspace,\\
{\mathcal B}^{2,n}:=&\,{\sf r}^{n+\half}\in\dspace,\\
{\mathcal B}^{3,n}:=&\,\iii\,\gff_{\ddelta}\big(\Phi_{\ddelta}^{n+\half}\big)
\otimes\left[\gamma_{\ddelta}\left(\tfrac{u^{n+1}+u^n}{2}\right)-\gamma_{\ddelta}
\left(\tfrac{V_{\ddelta}^{n+1}+V_{\ddelta}^n}{2}\right)\right]\in\dspace,\\
{\mathcal B}^{4,n}:=&\,\iii\,\left[\,
\gff_{\ddelta}\big(g(|\uu^{n+\half}|^2)\big)
-\gff_{\ddelta}\big(\Phi_{\ddelta}^{n+\half}\big)
\,\right]\otimes\tfrac{\uu^{n+1}+\uu^n}{2}\in\dspace\cdot\\
\end{split}
\end{equation}
%
%
%
%
%
%
\par\noindent\vskip0.2truecm\par\noindent
$\boxed{{\tt Part\,\,\,4}}:$ 
First, take the $(\cdot,\cdot)_{0,h}-$inner product of \eqref{ABBA_0}
with $(\vq^{n+1}+\vq^n)$, and then use \eqref{NewEra2}, keep
the real parts of the relation obtained and use the Cauchy-Schwarz inequality,
to arrive at
\begin{equation}\label{ABBA_2}
\|\vq^{n+1}\|_{0,h}-\|\vq^n\|_{0,h}
\leq\,\tau\,\sum_{\ell=1}^4\|{\mathcal B}^{\ell,n}\|_{0,h},
\,\quad n=0,\dots,N-1.
\end{equation}
\par
Let $n\in\{0,\dots,N-1\}$. In view of \eqref{ELP_6} and \eqref{SHELL_02},
we have
\begin{equation}\label{ABBA_3}
\|{\mathcal B}^{1,n}\|_{0,h}+\|{\mathcal B}^{2,n}\|_{0,h}
\leq\,C\,(\tau^2+h^2).
\end{equation}
After applying \eqref{BasicHot5}, \eqref{ELP_5}
and the mean value theorem, we get
\begin{equation}\label{ABBA_4}
\begin{split}
\|{\mathcal B}^{3,n}\|_{0,h}\leq&\,C_{\ddelta}
\,\left\|\gamma_{\ddelta}\left(\tfrac{u^{n+1}+u^n}{2}\right)
-\gamma_{\ddelta}\left(\tfrac{V_{\ddelta}^{n+1}+V_{\ddelta}^n}{2}\right)
\right\|_{0,h}\\
\leq&\,C_{\ddelta}\,\|e^{n+1}+e^n\|_{0,h}\\
\leq&\,C_{\ddelta}\,(\|\vq^{n+1}\|_{0,h}+\|\vq^n\|_{0,h}+\|\varrho^{n+1}\|_{0,h}
+\|\varrho^n\|_{0,h})\\
\leq&\,C_{\ddelta}\,(h^2+\|\vq^{n+1}\|_{0,h}+\|\vq^n\|_{0,h})
\end{split}
\end{equation}
and
\begin{equation}\label{ABBA_5}
\begin{split}
\|{\mathcal B}^{4,n}\|_{0,h}\leq&\,C\,
\left\|\gff_{\ddelta}\big(g(|\uu^{n+\half}|^2)\big)
-\gff_{\ddelta}\big(\Phi_{\ddelta}^{n+\half}\big)\right\|_{0,h}\\
\leq&\,C_{\ddelta}\,\|\emid^n\|_{0,h}.\\
\end{split}
\end{equation}
\par
From \eqref{ABBA_2}, \eqref{ABBA_3}, \eqref{ABBA_4} and \eqref{ABBA_5},
follows that
\begin{equation}\label{ABBA_6}
\|\vq^{n+1}\|_{0,h}-\|\vq^n\|_{0,h}\leq\,C_{\ddelta}\,\tau\,
\big(\,\tau^2+h^2+\|\emid^n\|_{0,h}
+\|\vq^{n+1}\|_{0,h}+\|\vq^n\|_{0,h}\,\big)
\end{equation}
for $n=0,\dots,N-1$.
%
%
%
\par\noindent\vskip0.2truecm\par\noindent
$\boxed{{\tt Part\,\,\,5}}:$ 
Since $\ddelta\ge u_{\ssy\max}$, subtracting \eqref{modCB3a} from \eqref{NORAD_22}
and using \eqref{MLF_3}, we obtain
\begin{equation}\label{ABBA_7}
\emid^{n}+\emid^{n-1}
=2\,\left[\,g(|\gamma_{\ddelta}(u^n)|^2)-g(|\gamma_{\ddelta}(V_{\ddelta}^n)|^2)\right]
+2\,{\sf r}^n,\quad n=1,\dots,N-1,
\end{equation}
which, easily, yields that
\begin{equation}\label{ABBA_8}
\emid^{n}-\emid^{n-2}
=2\,\sigma^n_{\ddelta}+2\,({\sf r}^n-{\sf r}^{n-1}),\quad n=2,\dots,N-1,
\end{equation}
where $\sigma^n_{\ddelta}\in\dspace$ defined by
\begin{equation}\label{sigma_def}
\sigma^n_{\ddelta}:=-\left[\,g(|\gamma_{\ddelta}(V_{\ddelta}^n)|^2)
-g(|\gamma_{\ddelta}(V_{\ddelta}^{n-1})|^2)
-g(|\gamma_{\ddelta}(u^n)|^2)+g(|\gamma_{\ddelta}(u^{n-1})|^2)\,\right].
\end{equation}
\par
Applying \eqref{MLF_4}, \eqref{MLF_3}, \eqref{BasicHot3} with (${\mathfrak g}=g$
and $\varepsilon=2\,\sup_{\ssy\rset}\mol_{\ddelta}^2$),
\eqref{Jalapeno2}, \eqref{Jalapeno1},
\eqref{BasicHot5} and \eqref{BasicHot6}, we obtain
\begin{equation*}
\begin{split}
\|\sigma^n_{\ddelta}\|_{0,h}
\leq&\,C_{\ddelta}\,\left[\,\||\gamma_{\ddelta}(V_{\ddelta}^n)|^2
-|\gamma_{\ddelta}(V_{\ddelta}^{n-1})|^2-|\gamma_{\ddelta}(u^n)|^2
+|\gamma_{\ddelta}(u^{n-1})|^2\|_{0,h}\right.\\
&\hskip1.5truecm\left.+\,\big||u^n|^2-|u^{n-1}|^2\big|_{\infty,h}
\,\left\|\,|\gamma_{\ddelta}(V^{n-1}_{\ddelta})|^2-|\gamma_{\ddelta}(u^{n-1})|^2\,\right\|_{0,h}\right]\\
\leq&\,C_{\ddelta}\,\left[\,\left\|\,\gamma_{\ddelta}(V_{\ddelta}^n)
-\gamma_{\ddelta}(V_{\ddelta}^{n-1})-\gamma_{\ddelta}(u^n)
+\gamma_{\ddelta}(u^{n-1})\,\right\|_{0,h}\right.\\
&\,\quad\quad\quad+|u^n-u^{n-1}|_{\infty,h}\,\|\gamma_{\ddelta}(V^{n-1}_{\ddelta})
-\gamma_{\ddelta}(u^{n-1})\|_{0,h}\\
&\quad\quad\quad
\left.+\,\tau\,
\,\left\|\,\gamma_{\ddelta}(V^{n-1}_{\ddelta})-\gamma_{\ddelta}(u^{n-1})\,\right\|_{0,h}\right]\\
\leq&\,C_{\ddelta}\,\left(\,\|e^n-e^{n-1}\|_{0,h}
+\tau\,\|e^{n-1}\|_{0,h}\,\right)\\
\leq&\,C_{\ddelta}\,\left(\,\|\varrho^n-\varrho^{n-1}\|_{0,h}
+\|\vq^n-\vq^{n-1}\|_{0,h}\right.\\
&\,\hskip2.0truecm\left.
+\tau\,\|\varrho^{n-1}\|_{0,h}+\tau\,\|\vq^{n-1}\|_{0,h}\,\right),
\quad n=2,\dots,N-1,\\
\end{split}
\end{equation*}
which, along with \eqref{ELP_5} and \eqref{ELP_6}, yields
\begin{equation}\label{ABBA_9}
\|\sigma^n_{\ddelta}\|_{0,h}\leq\,C_{\ddelta}\,\tau\,(h^2
+\|\partial\vq^n\|_{0,h}+\|\vq^{n-1}\|_{0,h}),
\quad n=2,\dots,N-1.
\end{equation}
Taking the $(\cdot,\cdot)_{0,h}$ inner product of
both sides of \eqref{ABBA_8}
with $(\emid^{n}+\emid^{n-2})$, and then the
Cauchy-Schwarz inequality, it follows that
\begin{equation*}
\|\emid^{n}\|_{0,h}^2-\|\emid^{n-2}\|^2_{0,h}\leq\,2\,
\left(\,\|\sigma^n_{\ddelta}\|_{0,h}
+\|{\sf r}^n-{\sf r}^{n-1}\|_{0,h}\,\right)
\,\|\emid^{n}+\emid^{n-2}\|_{0,h},\quad
n=2,\dots,N-1,
\end{equation*}
which, along with \eqref{SHELL_22} and \eqref{ABBA_9}, yields
\begin{equation}\label{ABBA_10}
\|\emid^{n}\|_{0,h}-\|\emid^{n-2}\|_{0,h}
\leq C_{\ddelta}\,\tau\,\left(\,\|\partial\vq^n\|_{0,h}
+\|\vq^{n-1}\|_{0,h}
+\tau^2+h^2\right),\quad n=2,\dots,N-1.
\end{equation}
%
%
\par\noindent\vskip0.2truecm\par\noindent
$\boxed{{\tt Part\,\,\,6}}:$ 
Since $\vq^0=0$, after setting $n=0$ in \eqref{ABBA_6},
we conclude that there exists a constant
${\sf C}_{1,\ddelta}>0$ such that
\begin{equation*}
\|\vq^1\|_{0,h}\leq\,{\sf C}_{1,\ddelta}\,\tau\,\left(\,\tau^2+h^2
+\|\emid^0\|_{0,h}+\|\vq^1\|_{0,h}\right).
\end{equation*}
Assuming that $\tau\,{\sf C}_{1,\ddelta}\leq\tfrac{1}{2}$, the latter inequality
yields
\begin{equation}\label{inest_1}
\|\vq^1\|_{0,h}\leq\,C_{\ddelta}\,\tau\,(\tau^2+h^2+\|\emid^0\|_{0,h}).
\end{equation}
Setting $n=1$ in \eqref{ABBA_7}, and using \eqref{SHELL_21},
\eqref{BasicHot1} (with ${\mathfrak g}=g$ and
$\varepsilon=2\,\sup_{\ssy\rset}|\gff_{\ddelta}|^2$),
\eqref{MLF_4}, \eqref{Jalapeno1} and \eqref{BasicHot5}, we have
\begin{equation*}
\begin{split}
\|\emid^1\|_{0,h}\leq&\,2\,\big\|g(|\gamma_{\ddelta}(u^{1})|^2)
-g\big(|\gamma_{\ddelta}(V_{\ddelta}^{1})|^2\big)\big\|_{0,h}
+\|\emid^0\|_{0,h}+C\,\tau^2\\
\leq&\,C\,\left[
\,\big\|\,|\gamma_{\ddelta}(u^{1})|^2
-|\gamma_{\ddelta}(V_{\ddelta}^{1})|^2\,\big\|_{0,h}
+\tau^2+\|\emid^0\|_{0,h}\,\right]\\
\leq&\,C_{\ddelta}\,\left[
\,\big\|\gamma_{\ddelta}(u^{1})-\gamma_{\ddelta}(V_{\ddelta}^{1})\big\|_{0,h}
+\tau^2+\|\emid^0\|_{0,h}\,\right]\\
\leq&\,C_{\ddelta}\,\left(\|e^1\|_{0,h}+\tau^2+\|\emid^0\|_{0,h}\right),\\
\end{split}
\end{equation*}
which, along with \eqref{ELP_5} and \eqref{inest_1}, yields
\begin{equation}\label{inest_2}
\begin{split}
\|\emid^1\|_{0,h}\leq&\,C_{\ddelta}\,(\|\varrho^1\|_{0,h}+\|\vq^1\|_{0,h}
+\tau^2+h^2+\|\emid^0\|_{0,h})\\
\leq&\,C_{\ddelta}\,(\tau^2+h^2+\|\emid^0\|_{0,h}).\\
\end{split}
\end{equation}
Now, set $n=1$ in \eqref{ABBA_6} and then use \eqref{inest_1}
and \eqref{inest_2}, to conclude that there exists a constant
${\sf C}_{2,\ddelta}\ge{\sf C}_{1,\ddelta}$ such that
\begin{equation*}
\|\vq^2\|_{0,h}\leq\,{\sf C}_{2,\ddelta}\,\tau\,(\|\vq^2\|_{0,h}+\tau^2+h^2+\|\emid^0\|_{0,h}),
\end{equation*}
from which, after assuming that $\tau\,{\sf C}_{2,\ddelta}\leq\tfrac{1}{2}$, we obtain
\begin{equation}\label{inest_3}
\|\vq^2\|_{0,h}\leq\,C_{\ddelta}\,\tau\,(\tau^2+h^2+\|\emid^0\|_{0,h}).
\end{equation}
\par  
Since $\vq^0=0$, we use \eqref{ABBA_0} (with $n=0$),
\eqref{ABBA_3}, \eqref{ABBA_4}, \eqref{ABBA_5}
and \eqref{inest_1}, to get
\begin{equation}\label{inest_4}
\begin{split}
\|A_h(\partial\vq^1)\|_{0,h}=&\,\tau^{-1}\,\|A_h(\vq^1)\|_{0,h}\\
\leq&\,\sum_{\ell=1}^4\|{\mathcal B}^{\ell,0}\|_{0,h}\\
\leq&\,C_{\ddelta}\,(\tau^2+h^2+\|\emid^0\|_{0,h}).\\
\end{split}
\end{equation}
Finally, in view of \eqref{ABBA_0} (with $n=1$), \eqref{inest_4}, 
\eqref{ABBA_3}, \eqref{ABBA_4}, \eqref{ABBA_5}, \eqref{inest_2},
\eqref{inest_3} and \eqref{inest_1}, we obtain
\begin{equation}\label{inest_5}
\begin{split}
\|A_h(\partial\vq^2)\|_{0,h}=&\,\tau^{-1}\,\left[\,
\|A_h(\vq^2)\|_{0,h}+\|A_h(\vq^1)\|_{0,h}\,\right]\\
\leq&\,C_{\ddelta}\tau^{-1}\,\left[\,\|T_h(\vq^1)\|_{0,h}
+\tau\,\sum_{\ell=1}^4\|{\mathcal B}^{\ell,1}\|_{0,h}+\tau\,(\tau^2+h^2+\|\emid^0\|_{0,h})\,\right]\\
\leq&\,C_{\ddelta}\tau^{-1}\,\left[\,\|2\,\vq^1-A_h(\vq^1)\|_{0,h}
+\tau\,(\tau^2+h^2+\|\emid^0\|_{0,h})\,\right]\\
\leq&\,C_{\ddelta}\tau^{-1}\,\left[\,2\,\|\vq^1\|_{0,h}+\|A_h(\vq^1)\|_{0,h}
+\tau\,(\tau^2+h^2+\|\emid^0\|_{0,h})\,\right]\\
\leq&\,C_{\ddelta}\,(\tau^2+h^2+\|\emid^0\|_{0,h}).\\
\end{split}
\end{equation}
%
%
\par\noindent\vskip0.2truecm\par\noindent
$\boxed{{\tt Part\,\,\,7}}:$ 
Since $\ddelta>\max\{g_{\ssy\max},u_{\ssy\max}\}$,
from \eqref{ABBA_0}, we, easily, conclude that
\begin{equation}\label{XABA_0}
\partial\vq^{n+1}-\partial\vq^{n-1}=\iii\,\tfrac{\tau}{2}
\,\Delta_h(\partial\vq^{n+1}+2\,\partial\vq^{n}
+\partial\vq^{n-1})
+\sum_{\ell=1}^6{\sf\Gamma}^{\ell,n},
\quad n=2,\dots,N-1,
\end{equation}
where
\begin{equation*}
\begin{split}
{\sf\Gamma}^{1,n}:=&\,
{\sf R}_h\left[\tfrac{\uu(t_{n+1},\cdot)-\uu(t_n,\cdot)-u(t_{n-1},\cdot)+u(t_{n-2},\cdot)}{\tau}\right]
-\tfrac{\uu^{n+1}-\uu^n-u^{n-1}+\uu^{n-2}}{\tau},\\
{\sf\Gamma}^{2,n}:=&\,{\sf r}^{n+\half}-{\sf r}^{(n-2)+\half},\\
{\sf\Gamma}^{3,n}:=&\,-\iii\,\mol_{\ddelta}\big(\Phi_{\ddelta}^{(n-2)+\half}\big)
\otimes{\sf\Gamma}^{3,n}_{\star},\\
{\sf\Gamma}^{3,n}_{\star}:=&\,
\gamma_{\ddelta}\left(\tfrac{V^{n+1}_{\ddelta}+V_{\ddelta}^n}{2}\right)
-\gamma_{\ddelta}\left(\tfrac{V_{\ddelta}^{n-1}+V_{\ddelta}^{n-2}}{2}\right)
-\gamma_{\ddelta}\left(\tfrac{\uu^{n+1}+\uu^n}{2}\right)
+\gamma_{\ddelta}\left(\tfrac{\uu^{n-1}+\uu^{n-2}}{2}\right),\\
{\sf\Gamma}^{4,n}:=&\,-\iii\,\left[\,
g(|\uu^{n+\half}|^2)-g(|\uu^{(n-2)+\half}|^2)\,\right]
\otimes\left[\,\gamma_{\ddelta}\left(\tfrac{V_{\ddelta}^{n+1}+V_{\ddelta}^{n}}{2}\right)
-\gamma_{\ddelta}\left(\tfrac{\uu^{n+1}+\uu^{n}}{2}\right)\,\right],\\
{\sf\Gamma}^{5,n}:=&\,-\iii\,\left[
\mol_{\ddelta}\big(\Phi_{\ddelta}^{(n-2)+\half}\big)
-\mol_{\ddelta}(g(|\uu^{(n-2)+\half}|^2))\right]
\otimes\left(\tfrac{\uu^{n+1}+\uu^n-\uu^{n-1}-\uu^{n-2}}{2}\right),\\
{\sf\Gamma}^{6,n}:=&\,-\iii\,
\gamma_{\ddelta}\left(\tfrac{V_{\ddelta}^{n+1}+V_{\ddelta}^{n}}{2}\right)
\otimes{\sf\Gamma}^{6,n}_{\star},\\
{\sf\Gamma}^{6,n}_{\star}:=&\,\mol_{\ddelta}\big(\Phi_{\ddelta}^{n+\half}\big)
-\mol_{\ddelta}\big(\Phi_{\ddelta}^{(n-2)+\half}\big)
-\mol_{\ddelta}(g(|\uu^{n+\half}|^2))
+\mol_{\ddelta}(g(|\uu^{(n-2)+\half}|^2)).\\
\end{split}
\end{equation*}
%
%
%
\par\noindent\vskip0.2truecm\par\noindent
$\boxed{{\tt Part\,\,\,8}}:$ Let $n\in\{2,\dots,N-1\}$.
Observing that
\begin{equation}\label{Tzimi_0}
u(t_{n+1},x)-u(t_n,x)-u(t_{n-1},x)+u(t_{n-2},x)
=\int_{0}^{\tau}\left(\int_{t_{n-2}+s}^{t_{n}+s}u_{tt}(s',x)\;ds'\right)\;ds
\end{equation}
and using \eqref{ELP_5}, \eqref{SHELL_03}, \eqref{BasicHot5}
and the mean value theorem, we obtain
\begin{equation}\label{Tzimi_1}
\begin{split}
\|{\sf\Gamma}^{1,n}\|_{0,h}\leq&\,\tfrac{C}{\tau}
\,\int_{0}^{\tau}\left(\int_{t_{n-2}+s}^{t_{n}+s}
\big\|{\sf R}_h[u_{tt}(s',\cdot)]-\PP[u_{tt}(s',\cdot)]
\big\|_{0,h}\;ds'\right)\;ds\\
\leq&\,C\,\tau\,h^2,
\end{split}
\end{equation}
\begin{equation}\label{Tzimi_2}
\begin{split}
\|{\sf\Gamma}^{2,n}\|_{0,h}\leq&\,\|{\sf r}^{n+\half}-{\sf r}^{(n-1)+\half}\|_{0,h}
+\|{\sf r}^{(n-1)+\half}-{\sf r}^{(n-2)+\half}\|_{0,h}\\
\leq&\,C\,\tau^3,\\
\end{split}
\end{equation}
\begin{equation}\label{Tzimi_3}
\begin{split}
\|{\sf\Gamma}^{4,n}\|_{0,h}\leq&\,C_{\ddelta}\,\tau\,\|e^{n+1}+e^n\|_{0,h}\\
\leq&\,C_{\ddelta}\,\tau
\,(\|\vq^{n+1}\|_{0,h}+\|\vq^n\|_{0,h}+\|\varrho^{n+1}\|_{0,h}
+\|\varrho^n\|_{0,h})\\
\leq&\,C_{\ddelta}\,\tau\,(h^2+\|\vq^{n+1}\|_{0,h}+\|\vq^n\|_{0,h})\\
\end{split}
\end{equation}
and
\begin{equation}\label{Tzimi_4}
\begin{split}
\|{\sf\Gamma}^{5,n}\|_{0,h}\leq\,C_{\ddelta}\,\tau\,\|\emid^{n-2}\|_{0,h}.
\end{split}
\end{equation}
Also, we apply \eqref{BasicHot6}, \eqref{ELP_5} and \eqref{ELP_6}, to get
\begin{equation}\label{Tzimi_5}
\begin{split}
\|{\sf\Gamma}^{3,n}\|_{0,h}\leq&\,C_{\ddelta}\,\|{\sf\Gamma}_{\star}^{3,n}\|_{0,h}\\
\leq&\,C_{\ddelta}\,\left(\,\|e^{n+1}+e^n-e^{n-1}-e^{n-2}\|_{0,h}
+\tau\,\|e^{n+1}+e^{n}\|_{0,h}\right)\\
\leq&\,C_{\ddelta}\,\tau\,\left(\,\|\varrho^{n+1}\|_{0,h}
+\|\varrho^{n}\|_{0,h}+\|\vq^{n+1}\|_{0,h}+\|\vq^{n}\|_{0,h}\,\right)\\
&+C_{\ddelta}\,\left(\,\|\varrho^{n+1}+\varrho^n-\varrho^{n-1}-\varrho^{n-2}\|_{0,h}
+\|\vq^{n+1}+\vq^n-\vq^{n-1}-\vq^{n-2}\|_{0,h}\,\right)\\
\leq&\,C_{\ddelta}\,\tau\,\left(\,h^2+\|\vq^{n+1}\|_{0,h}
+\|\vq^{n}\|_{0,h}\,\right)\\
&+C_{\ddelta}\,\left(\,\|\vq^{n+1}-\vq^n\|_{0,h}
+2\,\|\vq^n-\vq^{n-1}\|_{0,h}
+\|\vq^{n-1}-\vq^{n-2}\|_{0,h}\,\right)\\
&+C_{\ddelta}\,\left(\,\|\varrho^{n+1}-\varrho^n\|_{0,h}
+2\,\|\varrho^n-\varrho^{n-1}\|_{0,h}
+\|\varrho^{n-1}-\varrho^{n-2}\|_{0,h}\,\right)\\
\leq&\,C_{\ddelta}\,\tau\,\left(\,h^2+\|\vq^{n+1}\|_{0,h}
+\|\vq^{n}\|_{0,h}+\|\partial\vq^{n+1}\|_{0,h}
+\|\partial\vq^n\|_{0,h}+\|\partial\vq^{n-1}\|_{0,h}\,\right).\\
\end{split}
\end{equation}
Now, we use \eqref{MLF_4}, \eqref{BasicHot3}
(with ${\mathfrak g}=\gff_{\ddelta}$),
\eqref{ABBA_8}, \eqref{ABBA_9} and \eqref{SHELL_22},
to have
\begin{equation}\label{Tzimi_6}
\begin{split}
\|{\sf\Gamma}^{6,n}\|_{0,h}\leq&\,C_{\ddelta}\,\|{\sf\Gamma}_{\star}^{6,n}\|_{0,h}\\
\leq&\,C_{\ddelta}\,\left(\,\tau\,\|\emid^{n-2}\|_{0,h}
+\|\emid^n-\emid^{n-2}\|_{0,h}\right)\\
\leq&\,C_{\ddelta}\,\left(\,\tau\,\|\emid^{n-2}\|_{0,h}
+\|\sigma^n_{\ddelta}\|_{0,h}+\|{\sf r}^n-{\sf r}^{n-1}\|_{0,h}\right)\\
\leq&\,C_{\ddelta}\,\tau\left(\,\tau^2+h^2+\|\emid^{n-2}\|_{0,h}
+\|\partial\vq^n\|_{0,h}+\|\vq^{n-1}\|_{0,h}\right)\\
\leq&\,C_{\ddelta}\,\tau\left(\,\tau^2+h^2+\|\emid^{n-2}\|_{0,h}
+\|\partial\vq^n\|_{0,h}+\|\vq^{n}\|_{0,h}+\tau\,\|\partial\vq^n\|_{0,h}\right)\\
\leq&\,C_{\ddelta}\,\tau\left(\,\tau^2+h^2+\|\emid^{n-2}\|_{0,h}
+\|\partial\vq^n\|_{0,h}+\|\vq^{n}\|_{0,h}\right).\\
\end{split}
\end{equation}
Combining \eqref{Tzimi_1}, \eqref{Tzimi_2},\eqref{Tzimi_3},\eqref{Tzimi_4},
\eqref{Tzimi_5} and \eqref{Tzimi_6}, we arrive at
\begin{equation}\label{Tzimi_7}
\begin{split}
\sum_{\ell=1}^6\|{\sf\Gamma}^{\ell,n}\|_{0,h}\leq&\,C_{\ddelta}\,\tau\,\big(\,
\tau^2+h^2+\|\vq^{n+1}\|_{0,h}+\|\vq^{n}\|_{0,h}+\|\emid^{n-2}\|_{0,h}\\
&\hskip1.5truecm
+\|\partial\vq^{n+1}\|_{0,h}
+\|\partial\vq^n\|_{0,h}+\|\partial\vq^{n-1}\|_{0,h}\,\big).\\
\end{split}
\end{equation}
%
%
\par\noindent\vskip0.2truecm\par\noindent
$\boxed{{\tt Part\,\,\,9}}:$ 
Let $I_h:\dspace\mapsto\dspace$ be the identity operator,
$A_h:=I_h-\iii\,\tfrac{\tau}{2}\,\Delta_h$,
$T_h:=I_h+\iii\,\tfrac{\tau}{2}\,\Delta_h$
and $B_h:=A_h^{-1}T_h$. In view of Lemma~\ref{operator_lemma},
\eqref{XABA_0} is equivalent to 
\begin{equation*}
\partial\vq^{n+1}=
(B_h-I_h)(\partial\vq^{n})+B_h(\partial\vq^{n-1})
+\sum_{\ell=1}^6A_h^{-1}({\sf\Gamma}^{\ell,n}),
\quad n=2,\dots,N-1,
\end{equation*}
%
%
which can be written in a vector operational form (cf. \cite{Besse2})
as follows
\begin{equation}\label{SABA_0}
\left[
\begin{array}{l}
\partial\vq^{n+1}\cr
\hline
\partial\vq^n \cr
\end{array}\right]=G\,
\left[
\begin{array}{l}
\partial\vq^{n}\cr
\hline
\partial\vq^{n-1} \cr
\end{array}\right]
+\left[
\begin{array}{l}
F^n\cr
\hline
0 \cr
\end{array}\right],
\quad n=2,\dots,N-1,
\end{equation}
%
%
%
%
where
\begin{equation*}
G:=\left[\begin{array}{c|c}
B_h-I_h & B_h\cr
\hline
I_h& 0\cr
\end{array}\right]
\quad{\rm and}\quad
F^n:=\sum_{\ell=1}^6A_h^{-1}({\sf\Gamma}^{\ell,n}).
\end{equation*}
Then, a simple induction argument yields that
\begin{equation}\label{AGALAB_1}
\left[
\begin{array}{l}
\partial\vq^{m+1}\cr
\hline
\partial\vq^m \cr
\end{array}\right]=G^{m-1}\,
\left[
\begin{array}{l}
\partial\vq^{2}\cr
\hline
\partial\vq^{1} \cr
\end{array}\right]
+\sum_{\ell=2}^{m}G^{m-\ell}\,
\left[
\begin{array}{l}
F^{\ell}\cr
\hline
0 \cr
\end{array}\right],
\quad m=2,\dots,N-1,
\end{equation}
%
%
%
Observing that $G$ can be diagonalized as follows 
\begin{equation*}
G=\left[\begin{array}{c|c}
-I_h & I_h\cr
\hline
\ \ I_h&\ \ B_h^{-1}\cr
\end{array}\right]
\,\left[\begin{array}{c|c}
-I_h & 0\cr
\hline
\ \ 0& B_h\cr
\end{array}\right]
\,\left[\begin{array}{c|c}
-I_h & I_h\cr
\hline
\ \ I_h& B_h^{-1}\cr
\end{array}\right]^{-1}
\end{equation*}
we conclude that
\begin{equation}\label{AGALAB_2}
\begin{split}
G^{\kappa}=&\,\left[\begin{array}{c|c}
-I_h & I_h\cr
\hline
\ \ I_h&\ \ B_h^{-1}\cr
\end{array}\right]
\,\left[\begin{array}{c|c}
(-1)^{\kappa}\,I_h & 0\cr
\hline
\ \ 0& B_h^{\kappa}\cr
\end{array}\right]
\,\left[\begin{array}{c|c}
-I_h & I_h\cr
\hline
\ \ I_h& B_h^{-1}\cr
\end{array}\right]^{-1}\\
=&\,
\left[\begin{array}{c|c}
(-1)^{\kappa+1}\,I_h & B_h^{\kappa}\cr
\hline
(-1)^{\kappa}\,I_h&\ \ B_h^{\kappa-1}\cr
\end{array}\right]
\,\left[\begin{array}{c|c}
-I_h & I_h\cr
\hline
\ \ I_h& B_h^{-1}\cr
\end{array}\right]^{-1}
\quad\forall\,\kappa\in\nset.\\
\end{split}
\end{equation}
%
%
%
It is easily seen that
\begin{equation*}
\begin{split}
\left[
\begin{array}{c|c}
-I_h & I_h\cr
\hline
\ \ I_h & B_h^{-1}\cr
\end{array}\right]^{-1}=&\,\left[
\begin{array}{c|c}
-(I_h+B_h^{-1})^{-1}B_h^{-1} & (I_h+B_h^{-1})^{-1}\cr
\hline
\ \ (I_h+B_h^{-1})^{-1}& (I_h+B_h^{-1})^{-1}\cr
\end{array}\right]\\
=&\,\left[
\begin{array}{c|c}
-(I_h+B_h)^{-1} & B_h(I_h+B_h)^{-1}\cr
\hline
\ \ B_h(I_h+B_h)^{-1}& B_h(I_h+B_h)^{-1}\cr
\end{array}\right],\\
\end{split}
\end{equation*}
which, along with \eqref{AGALAB_2} and \eqref{Megatree4}, yields
\begin{equation}\label{AGALAB_3}
G^{\kappa}=\tfrac{1}{2}\,\left[
\begin{array}{c|c}
\left[(-1)^{\kappa}\,I_h+B_h^{\kappa+1}\right]\,A_h
&\left[(-1)^{\kappa+1}\,B_h+B_h^{\kappa+1}\right]\,A_h\cr
\hline
\left[(-1)^{\kappa+1}\,I_h+B_h^{\kappa}\right]\,A_h
&\left[(-1)^{\kappa}\,B_h+B_h^{\kappa}\right]\,A_h\cr
\end{array}\right]\quad\forall\,\kappa\in\nset.
\end{equation}
%
%
%
%
%
%
Combining \eqref{AGALAB_1} and \eqref{AGALAB_3}, we arrive at
\begin{equation}\label{AGALAB_4}
\begin{split}
\partial\vq^{m+1}=&\,\left[(-1)^{m-1}\,I_h+B_h^{m}\right]\,A_h(\partial\vq^2)
+\left[(-1)^{m}\,B_h+B_h^{m}\right]\,A_h(\partial\vq^1)\\
&\quad+\tfrac{1}{2}\,\sum_{\ell=2}^m\,
\left[(-1)^{m-\ell}\,I_h+B_h^{m-\ell+1}\right]
\,\left(\,\sum_{\ell'=1}^6{\sf\Gamma}^{\ell',\ell}\,\right)
\end{split}
\end{equation}
and
\begin{equation}\label{AGALAB_5}
\begin{split}
\partial\vq^{m}=&\,\left[(-1)^{m}\,I_h+B_h^{m-1}\right]\,A_h(\partial\vq^2)
+\left[(-1)^{m-1}\,B_h+B_h^{m-1}\right]\,A_h(\partial\vq^1)\\
&\quad+\tfrac{1}{2}\,\sum_{\ell=2}^m\,
\left[(-1)^{m-\ell+1}\,I_h+B_h^{m-\ell}\right]
\,\left(\,\sum_{\ell'=1}^6{\sf\Gamma}^{\ell',\ell}\,\right)
\end{split}
\end{equation}
for $m=2,\dots,N-1$.
%
%
%
\par\noindent\vskip0.2truecm\par\noindent
$\boxed{{\tt Part\,\,\,10}}:$
Applying the discrete norm $\|\cdot\|_{0,h}$ on both sides of
\eqref{AGALAB_4} and \eqref{AGALAB_5}, and then using
\eqref{Megatree1}, it follows that
\begin{equation*}
\begin{split}
\|\partial\vq^{m+1}\|_{0,h}+\|\partial\vq^{m}\|_{0,h}
\leq&\,4\,\left[\,\|A_h(\partial\vq^2)\|_{0,h}+\|A_h(\partial\vq^1)\|_{0,h}\,\right]\\
&\quad+2\,\sum_{\ell=2}^{m}\sum_{\ell=1}^6\|{\sf\Gamma}^{\ell',\ell}\|_{0,h},
\quad m=2,\dots,N-1,
\end{split}
\end{equation*}
which, along with \eqref{inest_4}, \eqref{inest_5} and
\eqref{Tzimi_7}, yields
\begin{equation}\label{Tzimi_8}
\begin{split}
\|\partial\vq^{m+1}\|_{0,h}+\|\partial\vq^{m}\|_{0,h}
\leq&\,C_{\ddelta}\,(\tau^2+h^2+\|\emid^0\|_{0,h}+\tau\,\|\vq^{m+1}\|_{0,h}
+\tau\,\|\partial\vq^{m+1}\|_{0,h})\\
&\,+C_{\ddelta}\,\tau\,\sum_{n=2}^m
(\|\emid^{n-2}\|_{0,h}+\|\vq^{n}\|_{0,h}
+\|\partial\vq^{n}\|_{0,h}
+\|\partial\vq^{n-1}\|_{0,h})
\end{split}
\end{equation}
for $m=2,\dots,N-1$.
\par
Observing that
$\|\vq^{n-1}\|_{0,h}\leq\tau\,\|\partial\vq^n\|+\|\vq^n\|_{0,h}$
and combining \eqref{ABBA_6} and \eqref{ABBA_10}, we obtain
\begin{equation*}
\begin{split}
\|\vq^{n+1}\|_{0,h}+\|\emid^{n}\|_{0,h}+\|\emid^{n-1}\|_{0,h}
\leq&\,\|\vq^n\|_{0,h}+\|\emid^{n-1}\|_{0,h}+\|\emid^{n-2}\|_{0,h}
+C_{\ddelta}\,\tau\,(\tau^2+h^2)\\
&\,\quad+C_{\ddelta}\,\tau\,\big(\|\vq^{n+1}\|_{0,h}+\|\vq^n\|_{0,h}
+\|\partial\vq^n\|_{0,h}+\|\emid^n\|_{0,h}\,\big)
\end{split}
\end{equation*}
for $n=2,\dots,N-1$. Now, sum both sides of the latter inequality with
respect to $n$ (from $2$ up to $m$)
and use \eqref{inest_3} and \eqref{inest_2}, to get
\begin{equation}\label{XE2019_1}
\begin{split}
\|\vq^{m+1}\|_{0,h}+\|\emid^{m}\|_{0,h}+\|\emid^{m-1}\|_{0,h}
\leq&\,C_{\ddelta}\,(\tau^2+h^2+\|\emid^0\|_{0,h}+\tau\,\|\vq^{m+1}\|_{0,h}
+\tau\,\|\emid^{m}\|_{0,h})\\
&\,\quad+C_{\ddelta}\,\tau\,\sum_{n=2}^m(\|\emid^{n-1}\|_{0,h}
+\|\vq^n\|_{0,h}+\|\partial\vq^n\|_{0,h})
\end{split}
\end{equation} 
for $m=2,\dots,N-1$.
Introducing the error quantities below
\begin{equation}\label{gammaerror}
\gamma_h^n:=\|\emid^{n-1}\|_{0,h}+\|\emid^{n-2}\|_{0,h}
+\|\vq^n\|_{0,h}+\|\partial\vq^n\|_{0,h}
+\|\partial\vq^{n-1}\|_{0,h},\quad n=2,\dots,N,
\end{equation}
from \eqref{Tzimi_8} and \eqref{XE2019_1} we conclude that there exists a
constant ${\sf C}_{3,\ddelta}\ge{\sf C}_{2,\ddelta}$ such that
\begin{equation}\label{XE2019_2}
(1-{\sf C}_{3,\ddelta})\,\gamma_h^{m+1}\leq\,{\sf C}_{3,\ddelta}\,(\tau^2+h^2+\|\emid^0\|_{0,h})
+C_{\ddelta}\,\tau\,\sum_{n=2}^m\gamma_h^n,
\quad m=2,\dots,N-1.
\end{equation}
Assuming that $\tau\,{\sf C}_{3,\ddelta}\leq\tfrac{1}{2}$ and applying a standard
discrete Gronwall argument, \eqref{XE2019_2} yields that
\begin{equation}\label{XE2019_3}
\max_{2\leq{m}\leq{\ssy N}}\gamma_h^m\leq\,C_{\ddelta}\,(\tau^2+h^2+\|\emid^0\|_{0,h}
+\gamma_h^2).
\end{equation}
Using \eqref{gammaerror}, \eqref{inest_2}, \eqref{inest_3}
and \eqref{inest_1}, we obtain
\begin{equation}\label{XE2019_4}
\begin{split}
\gamma_h^2=&\,\|\emid^{1}\|_{0,h}+\|\emid^{0}\|_{0,h}
+\|\vq^2\|_{0,h}+\|\partial\vq^2\|_{0,h}
+\|\partial\vq^{1}\|_{0,h}\\
\leq&\,C_{\ddelta}\,(\tau^2+h^2+\|\emid^0\|_{0,h})
+\tau^{-1}\,(\|\vq^2\|_{0,h}+2\,\|\vq^1\|_{0,h})\\
\leq&\,C_{\ddelta}\,(\tau^2+h^2+\|\emid^0\|_{0,h}).\\
\end{split}
\end{equation}
Thus, \eqref{XE2019_3}, \eqref{XE2019_4} and \eqref{inest_1}, yield
\begin{equation}\label{XE2019_5}
\max_{0\leq{m}\leq{\ssy N-1}}\|\emid^m\|_{0,h}
+\max_{0\leq{m}\leq{\ssy N}}\|\vq^m\|_{0,h}
+\max_{1\leq{m}\leq{\ssy N}}\|\partial\vq^m\|_{0,h}
\leq\,C_{\ddelta}\,(\tau^2+h^2+\|\emid^0\|_{0,h}).
\end{equation}
%
%
%
%
\par
Taking the $(\cdot,\cdot)_{0,h}-$inner product of \eqref{ABBA_0}
with $(\vq^{n+1}-\vq^n)$, and then using \eqref{NewEra1} and
keeping the imaginary parts of the relation obtained, we have
\begin{equation*}
|\vq^{n+1}|_{1,h}^2-|\vq^n|_{1,h}^2=2\,\tau\,\sum_{\ell=1}^4
\Imm[({\mathcal B}^{\ell,n},\partial\vq^{n+1})_{0,h}],\quad n=0,\dots,N-1,
\end{equation*}
which, along with the Cauchy-Schwarz inequality, 
\eqref{ABBA_3}, \eqref{ABBA_4}, \eqref{ABBA_5}
and \eqref{XE2019_5}, yields
\begin{equation}\label{XE2019_7}
\begin{split}
|\vq^{n+1}|_{1,h}^2-|\vq^n|_{1,h}^2
\leq&\,2\,\tau\,\|\partial\vq^{n+1}\|_{0,h}
\,\sum_{\ell=1}^4\|{\mathcal B}^{\ell,n}\|_{0,h}\\
\leq&\,C_{\ddelta}\,\tau\,(\tau^2+h^2+\|\emid^0\|_{0,h})^2,
\quad n=0,\dots,N-1.
\end{split}
\end{equation}
Since $\vq^0=0$, after applying a standard discrete Gronwall argument on
\eqref{XE2019_7}, we arrive at
\begin{equation}\label{XE2019_8}
\max_{0\leq{m}\leq{\ssy N}}|\vq^m|_{1,h}\leq\,C_{\ddelta}\,(\tau^2+h^2+\|\emid^0\|_{0,h}).
\end{equation}
Combining \eqref{XE2019_8}, \eqref{XE2019_5}, \eqref{ELP_5}
and \eqref{Les_Halles_4}, we obtain
\begin{equation}\label{Fight_1}
\begin{split}
\max_{0\leq{n}\leq{\ssy N}}\|e^n\|_{1,h}
\leq&\,\max_{0\leq{n}\leq{\ssy N}}
\left(\|\vq^n\|_{1,h}+\|\varrho^n\|_{1,h}\right)\\\
\leq&\,C_{\ddelta}\,(\tau^2+h^2+\|\emid^0\|_{0,h})\\
\leq&\,C_{\ddelta}\,(\tau^2+h^2),
\end{split}
\end{equation}
which, along with \eqref{Les_Halles_11}, \eqref{Les_Halles_3}
and \eqref{ELP_5}, establishes \eqref{mod_BR_cnv_1}.
%
%
%
\par\noindent\vskip0.2truecm\par\noindent
$\boxed{{\tt Part\,\,\,11}}:$
For $n=2,\dots,N-1$, let $\zeta^n_{\ddelta}\in\rspacez$
and $\psi^n_{\ddelta}\in\dspace$ be defined by
\begin{equation*}
\begin{split}
\zeta^n_{\ddelta}:=&\,|\gamma_{\ddelta}(V_{\ddelta}^n)|^2
-|\gamma_{\ddelta}(V_{\ddelta}^{n-1})|^2-|\gamma_{\ddelta}(u^n)|^2
+|\gamma_{\ddelta}(u^{n-1})|^2,\\
\psi^n_{\ddelta}:=&\,\gamma_{\ddelta}(V_{\ddelta}^n)
-\gamma_{\ddelta}(V_{\ddelta}^{n-1})-\gamma_{\ddelta}(u^n)
+\gamma_{\ddelta}(u^{n-1}).\\
\end{split}
\end{equation*}
Under the light of \eqref{sigma_def},
\eqref{BasicHot4} with (${\mathfrak g}=g$
and $\varepsilon=2\,\sup_{\ssy\rset}\mol_{\ddelta}^2$), \eqref{LH1},
\eqref{dpoincare}, \eqref{Jalapeno12}, \eqref{MLF_4}, 
\eqref{BasicHot5b}, \eqref{Fight_1}, \eqref{ELP_5},
\eqref{Jalapeno21}, \eqref{BasicHot7} and \eqref{ELP_6}, we have
\begin{equation}\label{ABBA_109}
\begin{split}
|\sigma^n_{\ddelta}|_{1,h}\leq&\,C_{\ddelta}\,\left[\,
\left(1+\big||\gamma_{\ddelta}(V^n_{\ddelta})|^2\big|_{1,h}
+\big||\gamma_{\ddelta}(V_{\ddelta}^{n-1})|^2\big|_{1,h}\right)
\,|\zeta^n_{\ddelta}|_{1,h}\right.\\
&\quad\quad\quad
\left.+\tau\,\left||\gamma_{\ddelta}(V^{n-1}_{\ddelta})|^2
-|\gamma_{\ddelta}(u^{n-1})|^2\right|_{1,h}\right]\\
\leq&\,C_{\ddelta}\,\left[\,
\left(1+|\gamma_{\ddelta}(V^n_{\ddelta})|_{\infty,h}\,
|\gamma_{\ddelta}(V^n_{\ddelta})|_{1,h}
+|\gamma_{\ddelta}(V_{\ddelta}^{n-1})|_{\infty,h}\,
|\gamma_{\ddelta}(V_{\ddelta}^{n-1})|_{1,h}\right)
\,|\zeta^n_{\ddelta}|_{1,h}\right.\\
&\quad\quad\quad
\left.+\tau\,\left(|\gamma_{\ddelta}(V_{\ddelta}^{n-1})|_{\infty,h}
+|\!|\!|\delta_h(u^{n-1})|\!|\!|_{\infty,h}\right)
\,\left|\gamma_{\ddelta}(V^{n-1}_{\ddelta})
-\gamma_{\ddelta}(u^{n-1})\right|_{1,h}\right]\\
\leq&\,C_{\ddelta}\,\left[\,
\left(1+|V^n_{\ddelta}|_{1,h}+|V_{\ddelta}^{n-1}|_{1,h}\right)
\,|\zeta^n_{\ddelta}|_{1,h}+\tau\,|e^{n-1}|_{1,h}\right]\\
\leq&\,C_{\ddelta}\,\left[\,
\left(1+|e^n|_{1,h}+|e^{n-1}|_{1,h}\right)
\,|\zeta^n_{\ddelta}|_{1,h}+\tau\,|e^{n-1}|_{1,h}\right]\\
\leq&\,C_{\ddelta}\,\left[\,|\zeta^n_{\ddelta}|_{1,h}
+\tau\,(\tau^2+h^2+\|\emid^0\|_{0,h})\right],\\
\end{split}
\end{equation}
\begin{equation}\label{ABBA_1010}
\begin{split}
|\zeta^n_{\ddelta}|_{1,h}\leq&\,C_{\ddelta}\,\left[
\,\left(1+|\gamma_{\ddelta}(V_{\ddelta}^n)|_{1,h}+|\gamma_{\ddelta}(V_{\ddelta}^{n-1})|_{1,h}\right)
\,|\psi^n_{\ddelta}|_{1,h}
+\tau\,|\gamma_{\ddelta}(V_{\ddelta}^{n-1})-\gamma_{\ddelta}(u^{n-1})|_{1,h}
\,\right]\\
\leq&\,C_{\ddelta}\,\left[
\,\left(1+|V_{\ddelta}^n|_{1,h}+|V_{\ddelta}^{n-1}|_{1,h}\right)
\,|\psi_{\ddelta}^n|_{1,h}+\tau\,|e^{n-1}|_{1,h}
\,\right]\\
\leq&\,C_{\ddelta}\,\left[
\,\left(1+|e^n|_{1,h}+|e^{n-1}|_{1,h}\right)
\,|\psi_{\ddelta}^n|_{1,h}
+\tau\,|e^{n-1}|_{1,h}\,\right]\\
\leq&\,C_{\ddelta}\,\left[\,|\psi^n_{\ddelta}|_{1,h}
+\tau\,(\tau^2+h^2+\|\emid^0\|_{0,h})\right]\\
\end{split}
\end{equation}
and
\begin{equation}\label{ABBA_1011}
\begin{split}
|\psi_{\ddelta}^n|_{1,h}\leq&\,C_{\ddelta}\,\left[
\left(1+|V_{\ddelta}^n|_{1,h}+|V_{\ddelta}^{n-1}|_{1,h}\right)\,|e^{n}-e^{n-1}|_{1,h}
+\tau\,|e^{n-1}|_{1,h}\right]\\
\leq&\,C_{\ddelta}\,\left[
\left(1+|e^n|_{1,h}+|e^{n-1}|_{1,h}\right)\,|e^{n}-e^{n-1}|_{1,h}
+\tau\,|e^{n-1}|_{1,h}\right]\\
\leq&\,C_{\ddelta}\,
\left[\,|\vq^{n}-\vq^{n-1}|_{1,h}+|\varrho^{n}-\varrho^{n-1}|_{1,h}
+\tau\,(\tau^2+h^2)\,\right]\\
\leq&\,C_{\ddelta}\,\tau\,
\left(\,|\partial\vq^{n}|_{1,h}+\tau^2+h^2+\|\emid^0\|_{0,h}\right)\\
\end{split}
\end{equation}
for $n=2,\dots,N-1$. Thus, \eqref{ABBA_109}, \eqref{ABBA_1010}
and \eqref{ABBA_1011}, yield that
\begin{equation}\label{ABBA_1012}
|\sigma^n_{\ddelta}|_{1,h}\leq\,C_{\ddelta}\,\tau\,
\left(\tau^2+h^2+\|\emid^0\|_{0,h}+|\partial\vq^n|_{1,h}\right),\quad n=2,\dots,N-1.
\end{equation}
Taking the $(\cdot,\cdot)_{0,h}$ inner product of both sides of \eqref{ABBA_8}
with $\Delta_h(\emid^{n}+\emid^{n-2})$, and then applying \eqref{NewEra1},
keeping the real part of the obtained relation and using
the Cauchy-Schwarz inequality, it follows that
\begin{equation*}
|\emid^{n}|_{1,h}^2-|\emid^{n-2}|^2_{1,h}\leq\,2\,
\left(\,|\sigma^n_{\ddelta}|_{1,h}
+|{\sf r}^n-{\sf r}^{n-1}|_{1,h}\,\right)
\,\left(|\emid^{n}|_{1,h}+|\emid^{n-2}|_{1,h}\right),
\quad n=2,\dots,N-1.
\end{equation*}
After using \eqref{SHELL_22H1} and \eqref{ABBA_1012},
the latter inequality yields
\begin{equation}\label{ABBA_1013}
|\emid^{n}|_{1,h}-|\emid^{n-2}|_{1,h}
\leq\,C_{\ddelta}\,\tau
\,(\tau^2+h^2+\|\emid^0\|_{0,h}+|\partial\vq^n|_{1,h}),
\quad n=2,\dots,N-1.
\end{equation}
Also, using \eqref{ABBA_7} (with $n=1$), \eqref{SHELL_21H1},
\eqref{BasicHot2} (with ${\mathfrak g}=g$ and
$\varepsilon=2\,\sup_{\ssy\rset}\mol_{\ddelta}^2$), \eqref{dpoincare},
\eqref{Jalapeno12}, \eqref{MLF_4} and \eqref{BasicHot5b}, we get
\begin{equation*}
\begin{split}
|\emid^1|_{1,h}\leq&\,C_{\ddelta}\,\left[|\emid^0|_{1,h}+|{\sf r}^1|_{1,h}
+\big|g(|\gamma_{\ddelta}(V_{\ddelta}^1)|^2)
-g(|\gamma_{\ddelta}(u^1)|^2)\big|_{1,h}\right]\\
\leq&\,C_{\ddelta}\,\left[\,\tau^2+|\emid^0|_{1,h}
+\big||\gamma_{\ddelta}(V_{\ddelta}^1)|^2
-|\gamma_{\ddelta}(u^1)|^2\big|_{1,h}\,\right]\\
\leq&\,C_{\ddelta}\,\left[\,\tau^2+|\emid^0|_{1,h}
+\big(1+|\gamma_{\ddelta}(V_{\ddelta}^1)|_{\infty,h}\big)\,\,\,
\big|\gamma_{\ddelta}(V_{\ddelta}^1)-\gamma_{\ddelta}(u^1)\big|_{1,h}\,\right]\\
\leq&\,C_{\ddelta}\,\left(\,\tau^2+|\emid^0|_{1,h}+|e^1|_{1,h}\,\right),\\
\end{split}
\end{equation*}
which, along with \eqref{Fight_1}, \eqref{Les_Halles_4} and 
\eqref{Les_Halles_4H1}, yields
\begin{equation}\label{Fight_2}
\begin{split}
|\emid^1|_{1,h}\leq&\,C_{\ddelta}\,(\tau^2+h^2+\|\emid^0\|_{1,h})\\
\leq&\,C_{\ddelta}\,(\tau^2+h^2).\\
\end{split}
\end{equation}
\par   
Summing both sides of \eqref{ABBA_1013}
with respect to $n$ (from $2$ up to $m$)
and using \eqref{Fight_2}, we arrive at
\begin{equation}\label{ABBA_Final}
|\emid^m|_{1,h}+|\emid^{m-1}|_{1,h}\leq\,
C_{\ddelta}\,(\tau^2+h^2+\|\emid^0\|_{1,h})+C_{\ddelta}\,\tau
\,\sum_{\ell=2}^m|\partial\vq^{\ell}|_{1,h},
\quad m=2,\dots,N-1.
\end{equation}
%
%
\par\noindent\vskip0.2truecm\par\noindent
$\boxed{{\tt Part\,\,\,12}}:$ First, we observe that
\eqref{XE2019_5}, \eqref{Les_Halles_4} and the inverse inequality \eqref{H1L2}
yield that
\begin{equation}\label{MidBound1}
\max_{0\leq{m}\leq{\ssy N-1}}|\emid^m|_{1,h}\leq\,C_{\ddelta}
\,\left[h+\left(\tfrac{\tau}{h}\right)\,\tau\right].
\end{equation}
From \eqref{ABBA_7}, \eqref{SHELL_21H1}, \eqref{BasicHot2},
\eqref{dpoincare}, \eqref{Jalapeno12} and \eqref{BasicHot5b},
we obtain
\begin{equation*}
\begin{split}
|\emid^n|_{1,h}\leq&\,2\,\big|g(|\gamma_{\ddelta}(V_{\ddelta}^n)|^2)
-g(|\gamma_{\ddelta}(u^n)|^2)\big|_{1,h}
+2\,|{\sf r}^n|_{1,h}+|\emid^{n-1}|_{1,h}\\
\leq&\,C_{\ddelta}\,\left[\tau^2
+\big||\gamma_{\ddelta}(V_{\ddelta}^n)|^2
-|\gamma_{\ddelta}(u^n)|^2
\big|_{1,h}\right]+|\emid^{n-1}|_{1,h}\\
\leq&\,C_{\ddelta}\,\left[\tau^2
+|\gamma_{\ddelta}(V_{\ddelta}^n)
-\gamma_{\ddelta}(u^n)|_{1,h}\right]+|\emid^{n-1}|_{1,h}\\
\leq&\,C_{\ddelta}\,(\tau^2+|e^n|_{1,h})+|\emid^{n-1}|_{1,h},
\quad n=1,\dots,N-1,
\end{split}
\end{equation*}
which, along with \eqref{Fight_1}, yields
\begin{equation*}
|\emid^n|_{1,h}\leq|\emid^{n-1}|_{1,h}+C_{\ddelta}\,(\tau^2+h^2),
\quad n=1,\dots,N-1.
\end{equation*}
Now, summing with respect to $n$ (from $1$ up to $m$)
and using \eqref{Les_Halles_4H1}, we get
\begin{equation}\label{MidBound2}
\begin{split}
\max_{0\leq{m}\leq{\ssy N-1}}|\emid^m|_{1,h}
\leq&\,|\emid^{0}|_{1,h}+C_{\ddelta}\,\tau^{-1}\,(\tau^2+h^2)\\
\leq&\,C_{\ddelta}\,\left[h\,\left(\tfrac{h}{\tau}\right)+\tau+\tau^2+h^2\right].\\
\end{split}
\end{equation}
Observing that $\min\{\tfrac{h}{\tau},\tfrac{\tau}{h}\}\leq1$ (cf. \cite{GLS})
and combining \eqref{MidBound1} and \eqref{MidBound2}, we conclude, easily, that
there exists a positive constant ${\sf C}_{\ddelta}^{\ssy\sf B,H^1}$ which 
is independent of $h$ and $\tau$, and such that
\begin{equation}\label{Global_MidBound}
\max_{0\leq{m}\leq{\ssy N-1}}|\emid^m|_{1,h}\leq\,{\sf C}_{\ddelta}^{\ssy\sf B, H^1}.
\end{equation} 
%
%
\par\noindent\vskip0.2truecm\par\noindent
$\boxed{{\tt Part\,\,\,13}}:$
Let $n\in\{2,\dots,N-1\}$. Using \eqref{Tzimi_0},
\eqref{ELP_5} and \eqref{SHELL_03H1}, we get
\begin{equation}\label{xantres_1}
|{\sf\Gamma}^{1,n}|_{1,h}+|{\sf\Gamma}^{2,n}|_{1,h}\leq\,C\,\tau\,(\tau^2+h^2).
\end{equation}
Also, using the mean value theorem, \eqref{LH1}, \eqref{BasicHot5b}, \eqref{dpoincare},
\eqref{Fight_1} and \eqref{BasicHot2} (with ${\mathfrak g}=\mol_{\ddelta}$),
we obtain
\begin{equation}\label{xantres_2}
\begin{split}
|{\sf\Gamma}^{4,n}|_{1,h}\leq&\,
\left|g(|\uu^{n+\half}|^2)-g(|\uu^{(n-2)+\half}|^2)\right|_{\infty,h}\,
\left|\gamma_{\ddelta}\left(\tfrac{V_{\ddelta}^{n+1}+V_{\ddelta}^{n}}{2}\right)
-\gamma_{\ddelta}\left(\tfrac{\uu^{n+1}+\uu^{n}}{2}\right)\right|_{1,h}\\
&\,\quad+\left|g(|\uu^{n+\half}|^2)-g(|\uu^{(n-2)+\half}|^2)\right|_{1,h}\,
\left|\gamma_{\ddelta}\left(\tfrac{V_{\ddelta}^{n+1}+V_{\ddelta}^{n}}{2}\right)
-\gamma_{\ddelta}\left(\tfrac{\uu^{n+1}+\uu^{n}}{2}\right)\right|_{\infty,h}\\
\leq&\,C\,\tau\,\left|\gamma_{\ddelta}\left(
\tfrac{V_{\ddelta}^{n+1}+V_{\ddelta}^{n}}{2}\right)
-\gamma_{\ddelta}\left(\tfrac{\uu^{n+1}+\uu^{n}}{2}\right)\right|_{1,h}\\
\leq&\,C_{\ddelta}\,\tau\,\left(1+|\!|\!|\delta_h(u^{n+1}+u^n)|\!|\!|_{\infty,h}\right)
\,|e^{n+1}+e^n|_{1,h}\\
\leq&\,C_{\ddelta}\,\tau\,(\tau^2+h^2+\|\emid^0\|_{0,h})\\
\end{split}
\end{equation}
and
\begin{equation}\label{xantres_3}
\begin{split}
|{\sf\Gamma}^{5,n}|_{1,h}\leq&\,
\left|(u^{n+1}+u^n)-(u^{n-1}+u^{n-1})\right|_{\infty,h}\,
\big|\mol_{\ddelta}(\Phi_{\ddelta}^{(n-2)+\half})
-\mol_{\ddelta}(g(|u^{(n-2)+\half}|^2))\big|_{1,h}\\
&\,\quad+\left|(u^{n+1}+u^n)-(u^{n-1}+u^{n-1})\right|_{1,h}\,
\big|\mol_{\ddelta}(\Phi_{\ddelta}^{(n-2)+\half})
-\mol_{\ddelta}(g(|u^{(n-2)+\half}|^2))\big|_{\infty,h}\\
\leq&\,C\,\tau\,\big|\mol_{\ddelta}(\Phi_{\ddelta}^{(n-2)+\half})
-\mol_{\ddelta}(g(|u^{(n-2)+\half}|^2))\big|_{1,h}\\
\leq&\,C_{\ddelta}\,\tau\,\left[|\emid^{n-2}|_{1,h}
+|\!|\!|\delta_h(g(|u^{(n-2)+\half}|^2))|\!|\!|_{\infty,h}\,
\|\emid^{n-2}\|_{0,h}\right]\\
\leq&\,C_{\ddelta}\,\tau\,|\emid^{n-2}|_{1,h}.\\
\end{split}
\end{equation}
Applying \eqref{BasicHot4} (with ${\mathfrak g}=\mol_{\ddelta}$),
\eqref{LH1}, \eqref{dpoincare}, \eqref{Global_MidBound},
\eqref{ABBA_8}, \eqref{ABBA_1012},
\eqref{SHELL_22H1}, \eqref{BasicHot7}, \eqref{Fight_1} and \eqref{ELP_6},
it follows that
\begin{equation}\label{xantres_4a}
\begin{split}
\big|{\sf\Gamma}_{\star}^{6,n}\big|_{1,h}\leq&\,C_{\ddelta}
\,\left[1+\big|\Phi_{\ddelta}^{n+\half}\big|_{1,h}
+\big|\Phi_{\ddelta}^{(n-2)+\half}\big|_{1,h}\right]
\,\,|\emid^n-\emid^{n-2}|_{1,h}
+C_{\ddelta}\,\tau\,|\emid^{n-2}|_{1,h}\\
\leq&\,C_{\ddelta}\,
\,\left(1+|\emid^n|_{1,h}+|\emid^{n-2}|_{1,h}\right)
\,\left(|\sigma^n_{\ddelta}|_{1,h}
+|{\sf r}^{n}-{\sf r}^{n-1}|_{1,h}\right)
+C_{\ddelta}\,\tau\,|\emid^{n-2}|_{1,h}\\
\leq&\,C_{\ddelta}\,\tau\,
\,\left(\tau^2+h^2+\|\emid^0\|_{0,h}+|\partial\vq^n|_{1,h}
+|\emid^{n-2}|_{1,h}\right)\\
\end{split}
\end{equation}
and
\begin{equation}\label{xantres_5a}
\begin{split}
|{\sf\Gamma}^{3,n}_{\star}|_{1,h}^2\leq&\,C_{\ddelta}\,\left(1+
|V_{\ddelta}^{n+1}+V_{\ddelta}^{n}|_{1,h}+|V_{\ddelta}^{n-1}+V_{\ddelta}^{n-2}|_{1,h}
\right)\,|e^{n+1}+e^n-e^{n-1}-e^{n-2}|_{1,h}\\
&\quad+C_{\ddelta}\,\tau\,|e^{n-1}+e^{n-2}|_{1,h}\\
\leq&\,C_{\ddelta}\,\left(1+
|e^{n+1}+e^{n}|_{1,h}+|e^{n-1}+e^{n-2}|_{1,h}\right)
\,|e^{n+1}+e^n-e^{n-1}-e^{n-2}|_{1,h}\\
&\quad+C_{\ddelta}\,\tau\,(\tau^2+h^2+\|\emid^0\|_{0,h})\\
\leq&\,C_{\ddelta}\,\left[
|\varrho^{n+1}-\varrho^n|_{1,h}+2\,|\varrho^{n}-\varrho^{n-1}|_{1,h}
+|\varrho^{n-1}-\varrho^{n-2}|_{1,h}\right.\\
&\hskip1.5truecm\left.+\tau\,|\partial\vq^{n+1}|_{1,h}+2\,\tau\,|\partial\vq^n|_{1,h}
+\tau\,|\partial\vq^{n-1}|_{1,h}+\tau\,(\tau^2+h^2+\|\emid^0\|_{0,h})\right]\\
\leq&\,C_{\ddelta}\,\tau\,\left(
|\partial\vq^{n+1}|_{1,h}+|\partial\vq^n|_{1,h}
+|\partial\vq^{n-1}|_{1,h}+\tau^2+h^2+\|\emid^0\|_{0,h}\right).\\
\end{split}
\end{equation}
Then, we use \eqref{MLF_4},
\eqref{LH1}, \eqref{BasicHot5b},
\eqref{Fight_1}, \eqref{xantres_4a}, \eqref{MLF_1},
\eqref{Global_MidBound} and \eqref{xantres_5a}, to get
\begin{equation}\label{xantres_4}
\begin{split}
|{\sf\Gamma}^{6,n}|_{1,h}\leq&\,
\left|\gamma_{\ddelta}\left(\tfrac{V_{\ddelta}^{n+1}+V_{\ddelta}^n}{2}\right)\right|_{\infty,h}
\,|{\sf\Gamma}_{\star}^{6,n}|_{1,h}
+\left|\gamma_{\ddelta}\left(\tfrac{V_{\ddelta}^{n+1}+V_{\ddelta}^n}{2}\right)\right|_{1,h}
\,|{\sf\Gamma}_{\star}^{6,n}|_{\infty,h}\\
\leq&\,C_{\ddelta}\,\left(1+\left|V_{\ddelta}^{n+1}+V_{\ddelta}^n\right|_{1,h}\right)
\,|{\sf\Gamma}_{\star}^{6,n}|_{1,h}\\
\leq&\,C_{\ddelta}\,|{\sf\Gamma}_{\star}^{6,n}|_{1,h}\\
\leq&\,C_{\ddelta}\,\tau\,
\,\left(\tau^2+h^2+\|\emid^0\|_{0,h}+|\partial\vq^n|_{1,h}
+|\emid^{n-2}|_{1,h}\right)
\end{split}
\end{equation}
and
\begin{equation}\label{xantres_5}
\begin{split}
|{\sf\Gamma}^{3,n}|_{1,h}\leq&\,
\left|\mol_{\ddelta}\big(\Phi_{\ddelta}^{(n-2)+\half}\big)\right|_{\infty,h}
\,|{\sf\Gamma}_{\star}^{3,n}|_{1,h}
+\left|\mol_{\ddelta}\big(\Phi_{\ddelta}^{(n-2)+\half}\big)\right|_{1,h}
\,|{\sf\Gamma}_{\star}^{3,n}|_{\infty,h}\\
\leq&\,C_{\ddelta}\,\left(1+\sup_{\ssy\rset}|\gff_{\ddelta}'|
\,|\Phi^{(n-2)+\half}_{\ddelta}|_{1,h}\right)
\,|{\sf\Gamma}_{\star}^{3,n}|_{1,h}\\
\leq&\,C_{\ddelta}\,|{\sf\Gamma}_{\star}^{3,n}|_{1,h}\\
\leq&\,C_{\ddelta}\,\tau\,\left(\tau^2+h^2+\|\emid^0\|_{0,h}
+|\partial\vq^{n+1}|_{1,h}+|\partial\vq^n|_{1,h}
+|\partial\vq^{n-1}|_{1,h}\right).\\
\end{split}
\end{equation}
Thus, from \eqref{xantres_1}, \eqref{xantres_2}, \eqref{xantres_3},
\eqref{xantres_4} and \eqref{xantres_5}, we arrive at
\begin{equation}\label{xantres_all}
\sum_{\ell=1}^6|{\sf\Gamma}^{\ell,n}|_{1,h}\leq\,C_{\ddelta}\,\tau\,
\left(\tau^2+h^2+\|\emid^0\|_{0,h}+|\partial\vq^{n+1}|_{1,h}+|\partial\vq^n|_{1,h}
+|\partial\vq^{n-1}|_{1,h}+|\emid^{n-2}|_{1,h}\right).
\end{equation}
%
%
\par\noindent\vskip0.2truecm\par\noindent
$\boxed{{\tt Part\,\,\,14}}:$
Let $n\in\{0,1\}$. Taking the $(\cdot,\cdot)_{0,h}-$inner
product of \eqref{ABBA_0} with $\Delta_h(\vq^{n+1}+\vq^{n})$
and then using \eqref{NewEra1} and keeping the real parts of
the relation obtained, it follows that
\begin{equation*}
|\vq^{n+1}|_{1,h}^2-|\vq^n|_{1,h}^2=\tau\,\sum_{\ell=1}^4\Ree[
(\!\!(\delta_h{\mathcal B}^{\ell,n},\delta_h(\vq^{n+1}+\vq^n))\!\!)_{0,h}],
\end{equation*}
which, along with the use of the Cauchy-Schwarz inequality, yields
\begin{equation}\label{Ah_0}
|\vq^{n+1}|_{1,h}-|\vq^n|_{1,h}\leq\tau\,\sum_{\ell=1}^4
|{\mathcal B}^{\ell,n}|_{1,h}.
\end{equation}
First, we observe that \eqref{ELP_6} and
\eqref{SHELL_02H1}, easily, yield the following bound
\begin{equation}\label{Ah_1}
|{\mathcal B}^{1,n}|_{1,h}+|{\mathcal B}^{2,n}|_{1,h}
\leq\,C\,(\tau^2+h^2).
\end{equation}
%
%
Combining \eqref{MLF_1}, \eqref{LH1}, \eqref{Global_MidBound},
\eqref{BasicHot5b}, \eqref{dpoincare}, \eqref{Fight_1},
\eqref{BasicHot2} (with ${\mathfrak g}=\mff_{\ddelta}$)
and \eqref{Fight_2}, we have
\begin{equation}\label{Ah_2}
\begin{split}
|{\mathcal B}^{3,n}|_{1,h}\leq&\,
|\gff_{\ddelta}\big(\Phi_{\ddelta}^{n+\half}\big)|_{\infty,h}\,
\Big|\gamma_{\ddelta}(\tfrac{u^{n+1}+u^{n}}{2})
-\gamma_{\ddelta}
\left(\tfrac{V_{\ddelta}^{n+1}+V_{\ddelta}^{n}}{2}\right)\Big|_{1,h}\\
&\quad+|\gff_{\ddelta}(\Phi^{n+\half}_{\ddelta})|_{1,h}\,
\Big|\gamma_{\ddelta}(\tfrac{u^{n+1}+u^{n}}{2})
-\gamma_{\ddelta}
\left(\tfrac{V_{\ddelta}^{n+1}+V_{\ddelta}^{n}}{2}\right)\Big|_{\infty,h}\\
\leq&\,C_{\ddelta}\,\tau\,\left[1+\sup_{\ssy\rset}|\gff_{\ddelta}'|
\,|\Phi^{n+\half}_{\ddelta}|_{1,h}\right]\,
\Big|\gamma_{\ddelta}
\left(\tfrac{V_{\ddelta}^{n+1}+V_{\ddelta}^{n}}{2}\right)
-\gamma_{\ddelta}(\tfrac{u^{n+1}+u^{n}}{2})\Big|_{1,h}\\
\leq&\,C_{\ddelta}\,|e^{n+1}+e^{n}|_{1,h}\\
\leq&\,C_{\ddelta}\,(\tau^2+h^2+\|\emid^0\|_{0,h})\\
\end{split}
\end{equation}
and
\begin{equation}\label{Ah_3}
\begin{split}
|{\mathcal B}^{4,n}|_{1,h}\leq&\,
|\uu^{n+1}+\uu^{n}|_{\infty,h}\,
\big|\gff_{\ddelta}(g(|\uu^{n+\half}|^2))
-\gff_{\ddelta}\big(\Phi_{\ddelta}^{n+\half}\big)\big|_{1,h}\\
&\hskip1.0truecm+|\!|\!|\delta_h(u^{n+1}+u^{n})|\!|\!|_{\infty,h}\,
\big\|\gff_{\ddelta}(g(|\uu^{n+\half}|^2))
-\gff_{\ddelta}\big(\Phi_{\ddelta}^{n+\half}\big)\big\|_{0,h}\\
\leq&\,C_{\ddelta}\,\left[\big|\gff_{\ddelta}\big(\Phi_{\ddelta}^{n+\half}\big)
-\gff_{\ddelta}(g(|\uu^{n+\half}|^2))\big|_{1,h}+\|\emid^{n}\|_{0,h}\right]\\
\leq&\,C_{\ddelta}\,(|\emid^n|_{1,h}+\|\emid^n\|_{0,h})\\
\leq&\,C\,|\emid^{n}|_{1,h}\\
\leq&\,C_{\ddelta}\,(\tau^2+h^2+\|\emid^0\|_{1,h}).\\
\end{split}
\end{equation}
\par
Since $\vq^0$, using \eqref{Ah_0}, \eqref{Ah_1}, \eqref{Ah_2} and \eqref{Ah_3},
we conclude that
\begin{equation}\label{init_H1}
|\vq^2|_{1,h}+|\vq^1|_{1,h}\leq\,C_{\delta}\,\tau\,(\tau^2+h^2+\|\emid^0\|_{1,h}).
\end{equation}
Combining \eqref{ABBA_0} (with $n=0$), \eqref{Ah_1}, \eqref{Ah_2},
\eqref{Ah_3}, we get
\begin{equation}\label{Ah_4}
\begin{split}
|A_h(\partial\vq^1)|_{1,h}=&\,\tau^{-1}\,|A_h(\vq^1)|_{1,h}\\
\leq&\,\sum_{\ell=1}^4|{\mathcal B}^{\ell,0}|_{1,h}\\
\leq&\,C_{\ddelta}\,(\tau^2+h^2+\|\emid^0\|_{1,h}).\\
\end{split}
\end{equation}
Finally, in view of \eqref{ABBA_0} (with $n=1$), \eqref{Ah_4}, 
\eqref{Ah_1}, \eqref{Ah_2}, \eqref{Ah_3} and \eqref{XE2019_8},
we obtain
\begin{equation}\label{Ah_5}
\begin{split}
|A_h(\partial\vq^2)|_{1,h}=&\,\tau^{-1}\,\left[\,
|A_h(\vq^2)|_{1,h}+|A_h(\vq^1)|_{1,h}\,\right]\\
\leq&\,C_{\ddelta}\tau^{-1}\,\left[\,|T_h(\vq^1)|_{1,h}
+\tau\,\sum_{\ell=1}^4|{\mathcal B}^{\ell,1}|_{1,h}+\tau\,(\tau^2+h^2+\|\emid^0\|_{1,h})\,\right]\\
\leq&\,C_{\ddelta}\tau^{-1}\,\left[\,|2\,\vq^1-A_h(\vq^1)|_{1,h}
+\tau\,(\tau^2+h^2+\|\emid^0\|_{1,h})\,\right]\\
\leq&\,C_{\ddelta}\tau^{-1}\,\left[\,2\,|\vq^1|_{1,h}+|A_h(\vq^1)|_{1,h}
+\tau\,(\tau^2+h^2+\|\emid^0\|_{1,h})\,\right]\\
\leq&\,C_{\ddelta}\,(\tau^2+h^2+\|\emid^0\|_{1,h}).\\
\end{split}
\end{equation}
%
%
%
\par\noindent\vskip0.2truecm\par\noindent
$\boxed{{\tt Part\,\,\,15}}:$
Apply the discrete norm $|\cdot|_{1,h}$ on both sides of
\eqref{AGALAB_4} and \eqref{AGALAB_5}, and then
use \eqref{Megatree5}, \eqref{Ah_4}, \eqref{Ah_5}
and \eqref{xantres_all}, to have
\begin{equation}\label{Round15_1}
\begin{split}
|\partial\vq^{m+1}|_{1,h}+|\partial\vq^{m}|_{1,h}
\leq&\,4\,\left[\,|A_h(\partial\vq^2)|_{1,h}+|A_h(\partial\vq^1)|_{1,h}\,\right]
+2\,\sum_{\ell=2}^{m}\sum_{\ell=1}^6|{\sf\Gamma}^{\ell',\ell}|_{1,h}\\
\leq&\,C_{\ddelta}\,\left(\tau^2+h^2+\|\emid^0\|_{1,h}+\tau\,|\partial\vq^{m+1}|_{1,h}\right)\\
&\quad
+\tau\,\sum_{\ell=2}^m\left(|\partial\vq^n|_{1,h}
+|\partial\vq^{n-1}|_{1,h}+|\emid^{n-2}|_{1,h}\right),
\quad m=2,\dots.N-1.\\
\end{split}
\end{equation}
Introducing the following discrete $H^1-$error quantities:
\begin{equation}\label{gammaH1}
{\widehat\gamma}_h^n:=|\emid^{n-1}|_{1,h}+|\emid^{n-2}|_{1,h}
+|\partial\vq^n|_{1,h}+|\partial\vq^{n-1}|_{1,h},\quad n=2,\dots,N,
\end{equation}
from \eqref{Round15_1} and \eqref{ABBA_Final} we conclude that there exists a
constant ${\sf C}_{4,\ddelta}\ge{\sf C}_{3,\ddelta}$ such that
\begin{equation}\label{Round15_2}
(1-{\sf C}_{4,\ddelta})\,{\widehat\gamma}_h^{m+1}
\leq\,{\sf C}_{4,\ddelta}\,(\tau^2+h^2+\|\emid^0\|_{1,h})
+C_{\ddelta}\,\tau\,\sum_{\ell=2}^m{\widehat\gamma}_h^{\ell},
\quad m=2,\dots,N-1.
\end{equation}
Assuming that $\tau\,{\sf C}_{4,\ddelta}\leq\tfrac{1}{2}$ and applying a standard
discrete Gronwall argument, \eqref{Round15_2} yields that
\begin{equation}\label{Round15_3}
\max_{2\leq{m}\leq{\ssy N}}{\widehat\gamma}_h^m
\leq\,C_{\ddelta}\,(\tau^2+h^2+\|\emid^0\|_{1,h}+{\widehat\gamma}_h^2).\\\
\end{equation}
Using \eqref{gammaH1}, \eqref{Fight_2} and \eqref{init_H1}, we obtain
\begin{equation}\label{Round15_4}
\begin{split}
{\widehat\gamma}_h^2=&\,|\emid^1|_{1,h}+|\emid^0|_{1,h}
+|\partial\vq^1|_{1,h}
+|\partial\vq^{2}|_{1,h}\\
\leq&\,C_{\ddelta}\,(\tau^2+h^2+\|\emid^0\|_{1,h})
+\tau^{-1}\,(|\vq^2|_{1,h}+2\,|\vq^1|_{1,h})\\
\leq&\,C_{\ddelta}\,(\tau^2+h^2+\|\emid^0\|_{1,h}).\\
\end{split}
\end{equation}
From, \eqref{Round15_3} and \eqref{Round15_4}, follows that
\begin{equation}\label{Round15_5}
\max_{0\leq{m}\leq{\ssy N-1}}|\emid^m|_{1,h}
+\max_{1\leq{m}\leq{\ssy N}}|\partial\vq^m|_{1,h}
\leq\,C_{\ddelta}\,(\tau^2+h^2+\|\emid^0\|_{1,h}).
\end{equation}
\par
Thus, \eqref{mod_BR_cnv_2} follows easily from \eqref{Round15_5},
\eqref{Les_Halles_4} and \eqref{Les_Halles_4H1}.
\end{proof}
%
%
Next, we present how the convergence result
of Theorem~\ref{modCB_Conv} changes when
$\Phi_{\ddelta}^{\frac{1}{2}}$ is a first order approximation
of $g(|u(t^{\half},\cdot)|^2)$.
%
%
\begin{thm}\label{mod_Conv_II}
Let $u_{\ssy\max}:=\max_{\ssy\QQ}|\uu|$,
$g_{\ssy\max}:=\max_{\ssy\QQ}|g(|\uu|^2)|$
and $\ddelta\ge\max\{u_{\ssy\max},g_{\ssy\max}\}$.
If 
\begin{equation}\label{Phi_init}
\Phi_{\ddelta}^{\frac{1}{2}}=g(|u^0|^2),
\end{equation}
then, there exist positive constants ${\widehat{\sf C}}^{\ssy{\sf A}}_{\ssy\ddelta}$
and ${\widehat{\sf C}}_{\ssy\ddelta}^{\ssy{\sf B}}$,
independent of $\tau$ and $h$, such that: if
$\tau\,{\widehat{\sf C}}^{\ssy{\sf A}}_{\ssy\ddelta}\leq\tfrac{1}{2}$, then
\begin{equation}\label{RFD_cnv_1}
\max_{0\leq{m}\leq{\ssy N-1}}\|g(\uu^{m+\half})-\Phi_{\ddelta}^{m+\half}\|_{1,h}
+\max_{0\leq{m}\leq{\ssy N}}\|\uu^m-V_{\ddelta}^m\|_{1,h}
\leq\,{\widehat{\sf C}}_{\ssy\ddelta}^{\ssy{\sf B}}\,(\tau+h^2)
\end{equation}
\end{thm}
%
%
\begin{proof}
Here, for simplicity, we keep the notation and the notation convection
of the proof of Theorem~\ref{modCB_Conv}.
\par
Under the choice \eqref{Phi_init},
we obtain $\|\emid^0\|_{1,h}=O(\tau)$ and hence \eqref{RFD_cnv_1}
follows, easily, moving along the lines of the proof of 
Theorem~\ref{modCB_Conv} (see \eqref{Fight_1} and \eqref{Round15_5}).
\end{proof}
%
\subsection{Convergence of the (RFD) method}\label{Section_50}
%
%
In this section we show how we can use the convergence results for the (MRFD) scheme
to conclude convergence of the (RFD) method.
%
%
\begin{thm}\label{DR_Final}
Let $u_{\ssy\max}:=\max_{\ssy Q}|\uu|$,
$g_{\ssy\max}:=\max_{\ssy Q}|g(\uu)|$,
$\ddelta\ge\,2\,\max\{u_{\ssy\max},g_{\ssy\max}\}$,
${\sf C}^{\ssy\sf A}_{\ddelta}$, ${\sf C}^{\ssy\sf B}_{\ddelta}$
and ${\sf C}^{\ssy\sf C}_{\ddelta}$ be the constants
specified in Theorem~\ref{modCB_Conv}. If
$\tau\,{\sf C}^{\ssy\sf A}_{\ddelta}\leq\tfrac{1}{2}$ and
\begin{equation}\label{REL_Xmesh}
\max\{{\sf C}_{\ddelta}^{\ssy\sf B},
{\sf C}_{\ddelta}^{\ssy\sf C}\}\,\sqrt{{\sf L}}\,(\tau^2+h^{2})
\leq\,\tfrac{\ddelta}{2},
\end{equation}
then
\begin{equation}\label{BR_true_1}
\max_{0\leq{m}\leq{\ssy N-1}}\|g(|\uu^{m+\half}|^2)-\Phi^{m+\half}\|_{1,h}
+\|u^{\half}-W^{\half}\|_1
+\max_{0\leq{m}\leq{\ssy N}}\|\uu^m-\UUU^m\|_{1,h}
\leq\,C\,(\tau^2+h^2).
\end{equation}
\end{thm}
%
%
%
%
\begin{proof}
The convergence estimates \eqref{mod_BR_cnv_1} and \eqref{mod_BR_cnv_2},
the mesh size condition \eqref{REL_Xmesh} and \eqref{LH1} imply that 
\begin{equation*}
\begin{split}
\max_{0\leq{m}\leq{\ssy N-1}}|\Phi^{m+\half}_{\ddelta}|_{\infty,h}
\leq&\,\max_{0\leq{m}\leq{\ssy N-1}}|g(|\uu^{m+\half}|^2)-
\Phi^{m+\half}_{\ddelta}|_{\infty,h}+g_{\ssy\max}\\
\leq&\,\sqrt{\sf L}\,
\max_{0\leq{m}\leq{\ssy N-1}}\|g(|\uu^{m+\half}|^2)-
\Phi^{m+\half}_{\ddelta}\|_{1,h}+\tfrac{\ddelta}{2}\\
\leq&\,\sqrt{{\sf L}}\,{\sf C}_{\ddelta}^{\ssy\sf C}
\,(\tau^2+h^2)+\tfrac{\ddelta}{2}\\
\leq&\,\ddelta
\end{split}
\end{equation*}
and
\begin{equation*}
\begin{split}
\max\big\{|V_{\ddelta}^{\half}|_{\infty,h},
\max_{0\leq{m}\leq{\ssy N}}|V^{m}_{\ddelta}|_{\infty,h}\big\}
\leq&\,\max\big\{|u^{\half}-V_{\ddelta}^{\half}|_{\infty,h},
\max_{0\leq{m}\leq{\ssy N}}|\uu^m-V_{\ddelta}^m|_{\infty,h}\big\}
+u_{\ssy\max}\\
\leq&\,\sqrt{\sf L}\,\max\big\{\|u^{\half}-V_{\ddelta}^{\half}\|_{1,h},
\max_{0\leq{m}\leq{\ssy N-1}}\|\uu^m-V_{\ddelta}^m\|_{1,h}\}+\tfrac{\ddelta}{2}\\
\leq&\,\sqrt{{\sf L}}\,{\sf C}_{\ddelta}^{\ssy\sf B}
\,(\tau^2+h^2)+\tfrac{\ddelta}{2}\\
\leq&\,\ddelta
\end{split}
\end{equation*}
which, along with \eqref{MLF_1} and \eqref{MLF_3},
yields
$\gamma_{\ddelta}(V_{\ddelta}^{\half})
=V_{\ddelta}^{\half}$, 
$\gff_{\delta}(\Phi_{\ddelta}^{m+\half})=\Phi_{\ddelta}^{m+\half}$
for $m=0,\dots,N-1$, and
$\gamma_{\ddelta}\left(\tfrac{V_{\ddelta}^{n+1}+V_{\ddelta}^n}{2}\right)
=\tfrac{V_{\ddelta}^{n+1}+V_{\ddelta}^n}{2}$ for $n=0,\dots,N-1$.
Thus, we conclude that for $\delta=\ddelta$ the (MRFD) approximations
are (RFD) approximations, i.e. $W^{\half}=\UUU^{\half}_{\ddelta}$,
$\UUU^n=V^n_{\ddelta}$ for $n=0,\dots,N$, and
$\Phi^{n+\half}=\Phi^{n+\half}_{\ddelta}$ for $n=0,\dots,N-1$.
%
%
Thus, we obtain \eqref{BR_true_1} as a simple outcome of \eqref{mod_BR_cnv_1}
and \eqref{mod_BR_cnv_2}.
\end{proof}
%
%
%
%
%
\begin{thm}\label{CNV_Final_2}
Let $u_{\ssy\max}:=\max_{\ssy Q}|\uu|$,
$g_{\ssy\max}:=\max_{\ssy Q}|g(\uu)|$,
$\ddelta\ge\,2\,\max\{u_{\ssy\max},g_{\ssy\max}\}$,
${\widehat{\sf C}}^{\ssy\sf A}_{\ddelta}$ and
${\widehat{\sf C}}^{\ssy\sf B}_{\ddelta}$ be the constants
specified in Theorem~\ref{mod_Conv_II},
$\Phi^{\half}=g(|u^0|^2)$, $\tau\,{\widehat{\sf C}}^{\ssy\sf A}_{\ddelta}\leq\tfrac{1}{2}$
and
${\widehat{\sf C}}_{\ddelta}^{\ssy\sf B}
\,\sqrt{{\sf L}}\,(\tau+h^{2})\leq\,\tfrac{\ddelta}{2}$.
Then, it holds that
\begin{equation}\label{kourelou_2}
\max_{0\leq{m}\leq{\ssy N-1}}\|g(|\uu^{m+\half}|^2)-\Phi^{m+\half}\|_{1,h}
\leq\,{\widehat{\sf C}}^{\ssy\sf B}_{\ddelta}\,(\tau+h^2).
\end{equation}
Also, there exists a positive constants
${\widehat{{\sf C}}}_{\ddelta}^{\ssy\sf D}\ge{\widehat{{\sf C}}}_{\ddelta}^{\ssy\sf A}$
and ${\widehat{{\sf C}}}_{\ddelta}^{\ssy\sf E}$,
independent of $\tau$ and $h$, such that: if
$\tau\,{\widehat{{\sf C}}}_{\ddelta}^{\ssy\sf D}\leq\half$, then
\begin{equation}\label{kourelou_3}
\max_{0\leq{m}\leq{\ssy N}}\|\uu^m-\UUU^m\|_{1,h}
\leq\,{\widehat{{\sf C}}}_{\ddelta}^{\ssy\sf E}\,(\tau^2+h^2).
\end{equation}
\end{thm}
%
%
%
%
%
\begin{proof}
Here, for simplicity, we keep the notation and the notation convection
of the proof of Theorem~\ref{modCB_Conv}.
\par
Using our assumptions and moving along the lines of the proof of Theorem~\ref{DR_Final},
we conclude that $\gamma_{\ddelta}(V_{\ddelta}^{\half})
=V_{\ddelta}^{\half}$, 
$\gff_{\delta}(\Phi_{\ddelta}^{m+\half})=\Phi_{\ddelta}^{m+\half}$
for $m=0,\dots,N-1$, and
$\gamma_{\ddelta}\left(\tfrac{V_{\ddelta}^{n+1}+V_{\ddelta}^n}{2}\right)
=\tfrac{V_{\ddelta}^{n+1}+V_{\ddelta}^n}{2}$ for $n=0,\dots,N-1$.
Thus, for $\delta=\ddelta$ the (MRFD) approximations
are (RFD) approximations, and \eqref{RFD_cnv_1} inherits
\eqref{kourelou_2} and 
\begin{equation}\label{kourelou_5}
\max_{0\leq{m}\leq{\ssy N}}\|\uu^m-\UUU^m\|_{1,h}
=\max_{0\leq{m}\leq{\ssy N}}\|\uu^m-V_{\ddelta}^m\|_{1,h}
\leq\,C_{\ddelta}\,(\tau+h^2).
\end{equation}
\par
Taking the $(\cdot,\cdot)_{0,h}-$inner
product of \eqref{ABBA_0} with $\Delta_h(\vq^{n+1}+\vq^{n})$
and then using \eqref{NewEra1} and keeping the real parts of
the relation obtained, it follows that
\begin{equation*}
|\vq^{n+1}|_{1,h}^2-|\vq^n|_{1,h}^2
=\sum_{\ell=1}^4{\mathcal Z}^{\ell,n},
\quad n=0,\dots,N-1,
\end{equation*}
where
\begin{equation*}
{\mathcal Z}^{\ell,n}:=\tau\,\Ree[
(\!\!(\delta_h{\mathcal B}^{\ell,n},\delta_h(\vq^{n+1}+\vq^n))\!\!)_{0,h}],
\quad\ell=1,2,3,4.
\end{equation*}
Then, we sum with respect to $n$ (from $n=0$ up to $n=m$) 
to get
\begin{equation}\label{kourelou_10}
|\vq^{m+1}|_{1,h}^2=\sum_{n=0}^m\sum_{\ell=1}^4
{\mathcal Z}^{\ell,n}, \quad m=0,\dots,N-1.
\end{equation}
In view of \eqref{Beta_Defs}, \eqref{ELP_6} and \eqref{SHELL_02H1},
after the application of the Cauchy-Schwarz inequality, we get
\begin{equation}\label{kourelou_11}
\begin{split}
\sum_{n=0}^m\left({\mathcal Z}^{1,n}+{\mathcal Z}^{2,n}\right)
\leq&\,C\,\tau\,\sum_{n=0}^m(\tau^2+h^2)\,(|\vq^{n+1}|_{1,h}+|\vq^n|_{1,h})\\
\leq&\,C\,\tau\,\sum_{n=0}^m
\left[(\tau^2+h^2)^2+|\vq^{n+1}|_{1,h}^2+|\vq^n|_{1,h}^2\right]\\
\leq&\,C\,
\left[(\tau^2+h^2)^2+\tau\,|\vq^{m+1}|_{1,h}^2
+\tau\,\sum_{n=0}^m|\vq^n|_{1,h}^2\right],\quad m=0,\dots,N-1.\\
\end{split}
\end{equation}
Combining \eqref{Beta_Defs}, \eqref{MLF_1}, \eqref{LH1}, 
\eqref{kourelou_2} and \eqref{ELP_5}, we have
\begin{equation}\label{kourelou_30}
\begin{split}
\sum_{n=0}^m{\mathcal Z}^{3,n}\leq&\,\tfrac{\tau}{2}\,\sum_{n=0}^m
\left[|\gff_{\ddelta}(\Phi_{\ddelta}^{n+\half})|_{\infty,h}\,
|e^{n+1}+e^{n}|_{1,h}\right.\\
&\hskip1.5truecm\left.+|\gff_{\ddelta}(\Phi^{n+\half}_{\ddelta})|_{1,h}\,
|e^{n+1}+e^{n}|_{\infty,h}\right]\,|\vq^{n+1}+\vq^n|_{1,h}\\
\leq&\,C_{\ddelta}\,\tau\,\sum_{n=0}^m\left[1+\sup_{\ssy\rset}|\gff_{\ddelta}'|
\,|\Phi^{n+\half}_{\ddelta}|_{1,h}\right]\,
(|e^{n+1}|_{1,h}+|e^{n}|_{1,h})\,(|\vq^{n+1}|_{1,h}+|\vq^n|_{1,h})\\
\leq&\,C_{\ddelta}\,\tau\,\sum_{n=0}^m\,(|\varrho^{n+1}|_{1,h}+|\varrho^n|_{1,h}
+|\vq^{n+1}|_{1,h}+|\vq^n|_{1,h})
\,(|\vq^{n+1}|_{1,h}+|\vq^n|_{1,h})\\
\leq&\,C_{\ddelta}\,\tau\,\sum_{n=0}^m\,(h^2+|\vq^{n+1}|_{1,h}+|\vq^n|_{1,h})
\,(|\vq^{n+1}|_{1,h}+|\vq^n|_{1,h})\\
\leq&\,C_{\ddelta}\,\left[h^4+\tau\,|\vq^{m+1}|_{1,h}^2
+\tau\,\sum_{n=0}^m|\vq^n|_{1,h}^2\right],\quad m=0,\dots,N-1.\\
\end{split}
\end{equation}
Now, from \eqref{Beta_Defs}, it follows that
\begin{equation}\label{kourelou_12}
\sum_{n=0}^m{\mathcal Z}^{4,n}={\mathcal Z}^{m}_{\ssy A}
+{\mathcal Z}^{m}_{\ssy B}+{\mathcal Z}^{m}_{\ssy C},
\quad m=0,\dots,N-1,
\end{equation}
where
\begin{equation*}
\begin{split}
{\mathcal Z}_{\ssy A}^m:=&\,-\tfrac{\tau}{2}
\,\Imm\left[(\!\!(\delta_h(\emid^{m}\otimes(u^{m+1}+u^{m})),
\delta_h\vq^{m+1})\!\!)_{0,h}\right],\\
{\mathcal Z}_{\ssy B}^m:=&\,-\tfrac{\tau}{2}
\,\Imm\left[\,
\sum_{n=1}^{m}(\!\!(\delta_h((\emid^{n-1}+\emid^n)\otimes(u^{n}+u^{n-1})),
\delta_h\vq^{n})\!\!)_{0,h}\,\right],\\
{\mathcal Z}_{\ssy C}^m:=&\,-\tfrac{\tau}{2}\,\Imm\left[\,
\sum_{n=1}^m(\!\!(\delta_h(\emid^n\otimes(u^{n+1}-u^{n-1})),
\delta_h\vq^{n})\!\!)_{0,h}\,\right].\\
\end{split}
\end{equation*}
Using the Cauchy-Schwarz inequality, \eqref{kourelou_2},
\eqref{ABBA_7}, \eqref{BasicHot1}, \eqref{BasicHot2}, 
\eqref{kourelou_5}, \eqref{LH1}, \eqref{SHELL_21}, 
\eqref{SHELL_21H1} and \eqref{ELP_5}, we have
\begin{equation}\label{kourelou_14}
\begin{split}
{\mathcal Z}_{\ssy A}^m\leq&\,\tfrac{\tau}{2}\,\left[\,
|\!|\!|\delta_h(u^m+u^{m+1})|\!|\!|_{\infty,h}\,\|\emid^m\|_{0,h}+
|u^{m+1}+u^m|_{\infty,h}\,|\emid^m|_{1,h}\right]
\,|\vq^{m+1}|_{1,h}\\
\leq&\,C\,\tau\,\|\emid^{m}\|_{1,h}\,|\vq^{m+1}|_{1,h}\\
\leq&\,C_{\ddelta}\,(\tau^2+\tau\,h^2)\,|\vq^{m+1}|_{1,h}\\
\leq&\,C_{\ddelta}\,(\tau^2+\tau\,h^2)^2+\tfrac{1}{2}
\,|\vq^{m+1}|_{1,h}^2,\quad m=0,\dots,N-1,\\
\end{split}
\end{equation}
\begin{equation}\label{kourelou_13}
\begin{split}
{\mathcal Z}_{\ssy C}^m\leq&\,\tfrac{\tau}{2}\,\sum_{n=1}^m
\left[\,
|\!|\!|\delta_h(u^{n+1}-u^{n-1})|\!|\!|_{\infty,h}\,\|\emid^n\|_{0,h}+
|u^{n+1}-u^{n-1}|_{\infty,h}\,|\emid^n|_{1,h}\right]\,|\vq^n|_{1,h}\\
\leq&\,C\,\tau\,\sum_{n=1}^m\tau\,\|\emid^n\|_{1,h}\,|\vq^n|_{1,h}\\
\leq&\,C_{\ddelta}\,\tau\,\sum_{n=0}^m(\tau^2+\tau\,h^2)\,|\vq^n|_{1,h}\\
\leq&\,C_{\ddelta}\,\left[(\tau^2+\tau\,h^2)^2
+\tau\,\sum_{n=0}^m|\vq^n|_{1,h}^2\right],
\quad m=1,\dots,N-1,\\
\end{split}
\end{equation}
and
\begin{equation}\label{kourelou_20}
\begin{split}
{\mathcal Z}_{\ssy B}^m\leq&\,C\,\tau\,\sum_{n=0}^m
\left(\,|\emid^n+\emid^{n-1}|_{1,h}+\|\emid^n+\emid^{n-1}\|_{0,h}\right)
\,\,|\vq^n|_{1,h}\\
\leq&\,C\,\tau\,\sum_{n=0}^m
\left(\,\|{\sf r}^n\|_{1,h}+\|g(|u^n|^2)-g(|V_{\ddelta}^n|^2)\|_{1,h}\right)
\,\,|\vq^n|_{1,h}\\
\leq&\,C_{\ddelta}\,\tau\,\sum_{n=0}^m(\,\tau^2+\|e^n\|_{1,h})\,|\vq^n|_{1,h}\\
\leq&\,C_{\ddelta}\,\tau\,\sum_{n=0}^m(\,\tau^2+h^2+|\vq^n|_{1,h})\,|\vq^n|_{1,h}\\
\leq&\,C_{\ddelta}\,\left[\,(\tau^2+h^2)^2
+\tau\,\sum_{n=0}^m|\vq^n|_{1,h}^2\,\right],
\quad m=1,\dots,N-1.\\
\end{split}
\end{equation}
\par
Now, from \eqref{kourelou_10}, \eqref{kourelou_11}, \eqref{kourelou_30},
\eqref{kourelou_12}, \eqref{kourelou_14}, \eqref{kourelou_13} and
\eqref{kourelou_20}, we conclude that there exists a positive constant
${\widehat{\sf C}}_{\ddelta}\ge\half\,{\widehat{\sf C}}_{\ddelta}^{\ssy A}$ such that
\begin{equation*}
(\tfrac{1}{2}-\tau\,{\widehat{\sf C}}_{\ddelta})\,|\vq^{m+1}|_{1,h}^2
\leq\,C_{\ddelta}\,\left[(\tau^2+h^2)^2+\tau\,\sum_{n=0}^m|\vq^n|_{1,h}^2\right],
\quad m=0,\dots,N-1.
\end{equation*}
Assuming that $2\,\tau\,{\widehat{\sf C}}_{\ddelta}\leq\tfrac{1}{2}$
and applying a discrete Gronwal argument we arrive at
\begin{equation*}\label{kourelou_80}
\max_{0\leq{m}\leq{\ssy N}}|\vq^{m}|_{1,h}
\leq\,C_{\ddelta}\,(\tau^2+h^2),
\end{equation*}
which, along with \eqref{ELP_5} and \eqref{dpoincare},
yields \eqref{kourelou_3}.
\end{proof}
%
%

%
\end{document}